\documentclass{amsart}
\usepackage{graphicx} 
\usepackage{amsmath,dsfont}
\usepackage{amssymb}
\usepackage{xcolor} 
\usepackage{mathtools}
\usepackage{colortbl} 
\usepackage{tikz} 
\usepackage{hf-tikz} 
\usepackage{makecell}
\usepackage{mathrsfs}
\usepackage{graphicx}
\usepackage{enumerate}
\usepackage{chngcntr}
\usepackage{enumitem}
\usepackage[colorlinks,
linkcolor=red,
anchorcolor=blue,
citecolor=green
]{hyperref}

\newtheorem{theorem}{Theorem}
\newtheorem*{theorem*}{Theorem}
\numberwithin{theorem}{section}
\newtheorem{definition}{Definition}[section]
\newtheorem{prop}{Proposition}[section]
\newtheorem{lemma}{Lemma}[section]
\newtheorem{remark}{Remark}[section]
\newtheorem{conjecture}{Conjecture}[section]

\DeclareMathOperator{\dist}{dist}
\DeclareMathOperator{\supp}{supp}

\newcommand{\diam}{\operatorname{diam}}

\title[Restriction estimates for surfaces with negative curvature]{Restriction estimates for surfaces with negative curvature in $\mathbb{R}^{3}$}
 \author{}
	\author{Shaoming Guo}
\address{Chern Institute of Mathematics, LPMC and New Cornerstone Science Laboratory,
Nankai University, Tianjin 300071, PR China}    \email{shaomingguo2018@gmail.com}
	\author{Diankun Liu}
	\address{School of Mathematical Sciences, Zhejiang University, Hangzhou 310027, PR China}
	\email{liudiankun5@gmail.com}
   
	\author{Yakun Xi}
	\address{School of Mathematical Sciences, Zhejiang University, Hangzhou 310027, PR China}
	\email{yakunxi@zju.edu.cn}

\begin{document}

\begin{abstract}

In $\mathbb{R}^{3}$, we prove that $L^q\to L^p$ restriction estimates associated with smooth surfaces with negative Gaussian curvature hold for all $p>\frac{22}{7}$ and $q'<\frac{p}{2}$. 
\end{abstract}

\maketitle

\section{Introduction}

Let
\[
\mathcal{S}=\{(\xi,\eta,h(\xi,\eta)):(\xi,\eta)\in[-1,1]^2\}.
\]
Define the extension operator associated with $\mathcal S$ by
\begin{equation}\label{eq extensionoperator}
E_{\mathcal{S}}f(x,y,t)
=\int_{[-1,1]^{2}}e^{i(x\xi+y\eta+t h(\xi,\eta))}\,f(\xi,\eta)\,d\xi\, d\eta,
\end{equation}
where $\mathcal{S}$ is a smooth surface with nonvanishing Gaussian curvature. Stein made the following conjecture.

\begin{conjecture}\label{conj}
Assume that $h$ is smooth and that the Gaussian curvature of $\mathcal S$ is nonzero at every point of $[-1,1]^2$. Then there exists a constant $C$, depending only on $p$, $q$, and the surface $\mathcal S$, such that
\begin{equation}\label{eq conj}
\|E_{\mathcal S}f\|_{L^{p}(\mathbb{R}^{3})}\le C\|f\|_{L^{q}([-1,1]^2)},
\end{equation}
for all $p>3$ and $q'\le \frac{p}{2}$, where $q'=\frac{q}{q-1}$.
\end{conjecture}

A great deal of work has been devoted to this conjecture. It was solved by Fefferman \cite{Fefferman} in $\mathbb{R}^{2}$, but remains open in all higher dimensions.

We briefly recall some background in $\mathbb R^3$. In the elliptic case, foundational works introduced many of the central techniques in restriction theory, including wave packet decomposition \cite{Bourgain}, bilinear restriction estimates \cite{Wolff01,Tao03}, the broad-narrow method \cite{BourgainGuth}, and polynomial partitioning \cite{Guth15}. More recently, Wang and Wu \cite{WangWu2024} proved restriction estimates in the range $p>\frac{22}{7}$ for $p=q$ by combining refined decoupling with incidence geometry, and Gao, Wu, and the third author \cite{GWX} extended this to the range $p>\frac{22}{7}$ and $q'<\frac{p}{2}$.  

The negatively curved case is more delicate. Although Stein's conjecture is expected to hold for all surfaces with nonvanishing Gaussian curvature, the hyperbolic setting exhibits substantially more complicated oscillatory behavior. One basic reason is that a surface with negative Gaussian curvature may contain straight lines, and along such directions the phase may fail to oscillate. The model example is the hyperbolic paraboloid
\[
\mathbb{H}=\{(\xi,\eta,\xi\eta):\xi,\eta\in[-1,1]^2\}.
\]
Indeed, through every point of $\mathbb H$, the coordinate curves in the $\xi$- and $\eta$-directions are straight lines on the surface. For a general surface with negative Gaussian curvature, the arrangement of such straight lines can be far more irregular, which makes the analysis substantially more difficult.

We briefly review previous results for surfaces with negative Gaussian curvature. For the model hyperbolic paraboloid $\mathbb H$, Lee \cite{Lee} and Vargas \cite{Vargas} independently proved the restriction estimate in the range $p>\frac{10}{3}$ and $q'<\frac{p}{2}$ by the bilinear method. This was later improved by Cho and Lee \cite{ChoLee}, Kim \cite{Kim}, and Stovall \cite{Stovall} to the range $p>3.25$ and $q'\le \frac{p}{2}$.

For more general negatively curved surfaces, Buschenhenke, M\"uller, and Vargas first studied the perturbed hyperbolic paraboloid
\[
h(\xi,\eta)=\xi\eta+\frac{\eta^3}{3},
\]
obtaining the range $p>\frac{10}{3}$ and $q'< \frac{p}{2}$ in \cite{BMV1}, and later improving this to $p>3.25$ and $q'< \frac{p}{2}$ in \cite{BMV2}. Subsequently, Buschenhenke, M\"uller, and Vargas \cite{BMV3} and the first author and Oh \cite{GO} independently proved that, for all smooth surfaces with negative Gaussian curvature, the restriction estimate holds in the range $p>3.25$ and $q'< \frac{p}{2}$. More recently, Demeter and Wu \cite{DW} proved the range $p=q>\frac{22}{7}$ for the model hyperbolic paraboloid by combining bilinear refined decoupling with incidence estimates.

We now state our main theorem. It extends the result of Demeter and Wu in two respects. First, we replace the model hyperbolic paraboloid by an arbitrary smooth surface with negative Gaussian curvature in $\mathbb R^3$. Second, for the same range $p>\frac{22}{7}$, we obtain $L^q\to L^p$ restriction estimates for all $q$ with $q'<\frac{p}{2}$.
\begin{theorem}\label{thm restriction}
Assume that $h$ is smooth and that the Gaussian curvature of $\mathcal S$ is negative at every point of $[-1,1]^2$. Then there exists a constant $C$, depending only on $p$, $q$, and the surface $\mathcal S$, such that
\begin{equation}
\|E_{\mathcal S}f\|_{L^{p}(\mathbb{R}^{3})}\le C\|f\|_{L^{q}([-1,1]^2)},
\end{equation}
for all $p>\frac{22}{7}$ and $q'<\frac{p}{2}$.
\end{theorem}
Our proof is inspired by the recent work of Demeter and Wu \cite{DW}, but significant new geometric difficulties arise for general smooth surfaces with negative Gaussian curvature. We extend the bilinear refined decoupling argument for the model hyperbolic paraboloid to polynomial perturbations satisfying a suitable $\iota$-curvature condition, and adapt the broad-narrow analysis to this more general setting.

The proof proceeds as follows. We first reduce the restriction problem for smooth surfaces with negative Gaussian curvature to a corresponding problem for a suitable class of polynomial surfaces, following the work of the first author and Oh \cite{GO}. For these polynomial surfaces, we classify lines in the parameter plane as either good or bad. Along good lines, we prove a bilinear refined decoupling inequality in a suitable generic setting, which serves as the main input in the broad case. Pairs with small curvature at both endpoints are covered by a bounded number of coarse squares and approximately affine strips. Outside these exceptional configurations, a combinatorial extraction produces a large subfamily with quantitative curvature at both endpoints. In the narrow case, the frequency support is concentrated in a coarse square or near an approximately affine strip. These contributions are treated by induction on scales using isotropic or anisotropic parabolic rescaling, respectively.

A central difficulty in extending the method of Demeter and Wu \cite{DW} from the model hyperbolic paraboloid to general smooth surfaces with negative Gaussian curvature lies in the bilinear refined decoupling step and the broad-narrow argument. In the model case $\mathbb H$, the degenerate directions are rigid: the bad lines occur only in the directions parallel to the $\xi$- and $\eta$-axes, and this rigid structure is compatible with the multiscale iteration. Moreover, the direction of vanishing normal curvature is tangent to a straight line on the surface, and along this line the curvature in the same direction vanishes everywhere. Therefore, we can perform the induction on scales directly along this bad line using anisotropic rescaling.

For a general negatively curved surface, however, the bad lines may be distributed in a much more irregular way, and it is no longer clear a priori how to propagate the geometric configuration needed for the iteration across intermediate scales. Moreover, at an arbitrary point, the direction corresponding to a vanishing normal curvature can vary quite irregularly. Indeed, the lift of a straight line in such a direction is generally a polynomial curve whose curvature need not vanish identically. Therefore, such a contribution cannot be directly treated as the narrow part in an induction-on-scales argument.

Given two points $p_1,p_2$ in the parameter domain of $\mathcal S=\{(p,h(p)):p\in[-1,1]^2\}$, we associate two geometric quantities to the pair $(p_1,p_2)$. The first is the line segment joining $p_{1}$ and $p_{2}$ in the parameter plane, which we denote by $l(p_{1},p_{2})$. The second is the unit vector, in the parameter plane, in the direction of the projection of the line of intersection of the tangent planes to $\mathcal S$ at $(p_{1},h(p_{1}))$ and $(p_{2},h(p_{2}))$. We denote this vector by $v(p_{1},p_{2})$.

We use a broad-narrow argument that leads to three cases. In the first two, the frequency support is concentrated either in an approximately affine strip or inside squares at a smaller scale, and the contribution is treated by induction on scales. In the third, we obtain two frequency pieces, each of which has a curvature lower bound $\iota$ in the direction $v(p_{1},p_{2})$. For two such squares, we establish an $\iota$-bilinear refined decoupling estimate.

In the proof of this bilinear refined decoupling estimate, the key new geometric input is that, under a suitable perturbative condition, the curve $h(l(p_{1},p_{2}))$ has curvature
$
\gtrsim \iota,
$
which in turn implies that, at the points $p_{1}$ and $p_{2}$, the normal curvature in the $v(p_{1},p_{2})$-direction is also bounded below by
$
\gtrsim \iota.
$
This perturbative condition is preserved throughout the multiscale iteration. In addition, in Section 3 we develop the geometry of degenerate directions and degenerate pairs of points for general polynomial perturbations of hyperbolic paraboloid surfaces in $\mathbb{R}^{3}$.

\subsection*{Statement on AI use}

ChatGPT was used for polishing and proofreading the paper. ChatGPT 6 was also used in developing the proof of Lemma 5.8.

\subsection*{Acknowledgment}  This project is supported by the National Key Research and Development Program of China No. 2022YFA1007200. S. G. is partly supported by the Nankai Zhide Foundation, NSFC Grant No. 12426204, and the New Cornerstone Science Foundation. D. L. and Y. X. are supported by Zhejiang Provincial Natural Science Foundation of China under
Grant No. LR25A010001, and NSFC Grant No.
12571107. The authors would like to thank Shukun Wu for suggesting this project and for helpful discussions on its main difficulties. The authors thank an anonymous referee and Sam Craig for pointing out an incompleteness in an earlier version of the proof.

\subsection*{Notation} 
\begin{enumerate}
\item We write $a \lessapprox b$ to mean that for any $\varepsilon > 0$ and any large parameter $P \geq 1$ arising from the context (typically $R$ or $\delta^{-1}$), there exists a constant $C_{\varepsilon}$ such that $a \leq C_{\varepsilon} P^{\varepsilon} b$.
    \item We write $a \gtrapprox b$ to mean that for any $\varepsilon > 0$ and any large parameter $P \geq 1$ arising from the context, there exists a constant $C_{\varepsilon}$ such that $a \geq C_{\varepsilon} P^{-\varepsilon} b$.
    \item We use $B^m_r(x)$ to denote the $m$-dimensional ball of radius $r$ centered at $x$; we may simply write $B^m_r$ or $B_r$ if $x$ and $m$ are clear from context.
    \item We use the terms \emph{strip} and \emph{rectangle} interchangeably to refer to a rectangular region whose aspect ratio is sufficiently large. In this paper, we typically refer to such a rectangle at a large scale as a \emph{strip}, and at a small scale as a \emph{rectangle}.
\end{enumerate}

\section{Standard reduction}
To prove Theorem \ref{thm restriction}, it suffices, by the $\varepsilon$-removal lemma of Tao \cite{Tao99}, to establish the following local estimate.

\begin{theorem}\label{thm restriction2}
For any $\epsilon>0$, there exists a constant $C_\epsilon$, depending only on $\epsilon$, $p$, $q$, and the surface $\mathcal S$, such that
\begin{equation}
\|E_{\mathcal S}f\|_{L^p(B_R)}\le C_\epsilon R^{\epsilon} \|f\|_{L^q([-1,1]^2)}
\end{equation}
for all $p\ge \frac{22}{7}$ and $q'\leq\frac{p}{2}$.
\end{theorem}

We consider surfaces of the form
\[
\mathcal{S}=\{(\xi,\eta,h(\xi,\eta)):(\xi,\eta)\in[-1,1]^2\}.
\]
By a reduction argument introduced in \cite[Section 2]{GO}, it suffices to treat polynomial phase functions of the form
\begin{equation}\label{eq reduction}
h(\xi,\eta)=\xi\eta+a_{2,0}\xi^2+a_{0,2}\eta^2+\sum_{i=3}^D\sum_{j=0}^i a_{i-j,j}\,\xi^{i-j}\eta^j,
\end{equation}
whose coefficients satisfy
\begin{equation}\label{eq smallcoeff}
|a_{2,0}|+|a_{0,2}|+100^D\sum_{i=3}^D\sum_{j=0}^i |a_{i-j,j}|\le \varepsilon_0
\end{equation}
for some constant $\varepsilon_0>0$ that can be chosen to be sufficiently small.

Thus the reduced surface is a small polynomial perturbation of the hyperbolic paraboloid
\[
\{(\xi,\eta,\xi\eta):(\xi,\eta)\in[-1,1]^2\}.
\]
Throughout the paper, we shall assume that $h$ is a polynomial of the form \eqref{eq reduction} that satisfies \eqref{eq smallcoeff}.

\section{Geometry of bad lines and transversality}
We begin by studying lines in the parameter plane $[-1,1]^2$. Their images under the graph map
\[
(\xi,\eta)\mapsto (\xi,\eta,h(\xi,\eta))
\]
play an important role in the geometric analysis of the surface.

For any line $l$ in the parameter plane, either the angle between $l$ and the $\xi$-axis is less than $\frac{\pi}{4}$, or the angle between $l$ and the $\eta$-axis is less than $\frac{\pi}{4}$.

If the angle between $l$ and the $\xi$-axis is less than $\frac{\pi}{4}$, then $l$ can be written as
\[
\eta=\tan\alpha\cdot \xi+a.
\]
In this case, we let $F_{\alpha,a}$ denote the affine transformation given by translation by $(0,-a)$ followed by rotation through angle $-\alpha$. Then $F_{\alpha,a}$ sends $l$ to the line $\{\eta=0\}$.

If the angle between $l$ and the $\eta$-axis is less than $\frac{\pi}{4}$, then $l$ can be written as
\[
\xi=\tan\alpha\cdot \eta+a.
\]
In this case, we let $M_{\alpha,a}$ denote the affine transformation given by translation by $(-a,0)$ followed by rotation sending the direction of $l$ to the direction of $(0,1)$. Then $M_{\alpha,a}$ sends $l$ to the line $\{\xi=0\}$.

We now isolate the geometric obstruction relevant to bilinear decoupling, namely those lines in the parameter plane whose images on the surface are straight lines. This motivates the following definition.

\begin{definition}[Bad lines]\label{def badline}
A line $l$ in the parameter plane is called a \textbf{bad line} if one of the following holds.

\begin{enumerate}
\item[(i)] If the angle between $l$ and the $\xi$-axis is less than $\frac{\pi}{4}$, then
\[
h\circ F_{\alpha,a}^{-1}(\xi,\eta)=\eta\,\tilde h(\xi,\eta)+L(\xi),
\]
where $\tilde h(\xi,\eta)$ is a polynomial and $L(\xi)$ is a linear function of $\xi$.

\item[(ii)] If the angle between $l$ and the $\eta$-axis is less than $\frac{\pi}{4}$, then
\[
h\circ M_{\alpha,a}^{-1}(\xi,\eta)=\xi\,\tilde h(\xi,\eta)+L(\eta),
\]
where $\tilde h(\xi,\eta)$ is a polynomial and $L(\eta)$ is a linear function of $\eta$.
\end{enumerate}
\end{definition}

{
\begin{definition}[The family of bad lines]
We define the almost horizontal family consisting of bad lines
$$\mathcal{B_L}=\{(k,b):\eta=k\xi+b \text{ is a bad line for the polynomial map $h$} \}.$$
The vertical case is the same.
\end{definition}

\begin{definition}[$\iota$-Good line segment]\label{def goodline}
We call a line segment $l$ in the parameter plane \textbf{$\iota$-good} if the curvature of $h(l)$ has a uniform lower bound $\iota$.
\end{definition}
}
Note that in our definition, a good line is a strictly stronger notion than a line that is merely not bad. We will prove in Proposition \ref{prop bad2} that every bad line makes angle at most $O(\varepsilon_0)$ with either the $\xi$-axis or the $\eta$-axis. 
 The significance of Definition \ref{def badline} is that bad lines are precisely the lines along which the restriction of $h$ is affine. The following lemma makes this correspondence precise.
\begin{lemma}
Assume that $h$ is a polynomial. Then a straight line segment $\Gamma$ is contained in $\mathcal S$ if and only if its projection onto the parameter plane is contained in a bad line in the sense of Definition \ref{def badline}.
\end{lemma}

\begin{proof}
We first prove the ``only if'' direction. Let $\Gamma\subset \mathcal S$ be a nontrivial straight line segment. Since $\mathcal S$ is a graph over the $(\xi,\eta)$-plane, the projection of $\Gamma$ onto the parameter plane cannot collapse to a point. Hence its projection is a line segment contained in an affine line $l$ in the parameter plane.

Assume first that the angle between $l$ and the $\xi$-axis is less than $\frac{\pi}{4}$. After applying the affine change of coordinates $F_{\alpha,a}$ introduced before Definition \ref{def badline}, we may assume that $l$ is given by $\eta=0$. Since $\Gamma$ is a straight line segment in $\mathbb R^3$, the restriction of the graphing function to this line is affine in $\xi$. Therefore there exists a linear function $L(\xi)$ such that
\[
(h\circ F_{\alpha,a}^{-1})(\xi,0)=L(\xi).
\]
It follows that
\[
(h\circ F_{\alpha,a}^{-1})(\xi,\eta)-L(\xi)
\]
vanishes on $\{\eta=0\}$, and hence is divisible by $\eta$. Since $h$ is a polynomial, we may write
\[
h\circ F_{\alpha,a}^{-1}(\xi,\eta)=\eta\,\widetilde h(\xi,\eta)+L(\xi)
\]
for some polynomial $\widetilde h$. Thus $l$ is a bad line by Definition \ref{def badline}(i).

If instead the angle between $l$ and the $\eta$-axis is less than $\frac{\pi}{4}$, then the same argument, using $M_{\alpha,a}$, shows that
\[
h\circ M_{\alpha,a}^{-1}(\xi,\eta)=\xi\,\widetilde h(\xi,\eta)+L(\eta),
\]
so $l$ is a bad line by Definition \ref{def badline}(ii).

Conversely, suppose that $l$ is a bad line. If $l$ is of the first type, then by Definition \ref{def badline}(i),
\[
h\circ F_{\alpha,a}^{-1}(\xi,\eta)=\eta\,\widetilde h(\xi,\eta)+L(\xi).
\]
Restricting to $\eta=0$, we obtain
\[
h\circ F_{\alpha,a}^{-1}(\xi,0)=L(\xi).
\]
Hence the image of the line $\{\eta=0\}$ under the graph map for $h\circ F_{\alpha,a}^{-1}$ is
\[
\{(\xi,0,L(\xi))\},
\]
which is a straight line in $\mathbb R^3$. Therefore the corresponding portion of the original surface $\mathcal S$ is also a straight line segment. The second type is treated in the same way.

This proves the lemma.
\end{proof}

Once bad lines have been characterized geometrically, the next step is to understand how many such lines may pass through a fixed point. The following lemma gives the required finiteness property.

{
\begin{lemma}\label{lem finite}
    Let $h$ be a polynomial of degree $D\ge 2$. Then there are at most $D$ possible directions of bad lines in the parameter plane. In particular, there exists a constant $C_D>0$, depending only on $D$, such that for every point $p\in[-1,1]^2$, the number of bad lines passing through $p$ is at most $C_D$.
\end{lemma} 
\begin{proof}
    Write
    $$
    h=h_D+h_{D-1}+\cdots+h_0,
    $$
    where $h_j$ is homogeneous of degree $j$, and $h_D\not\equiv 0$.

    Let $\ell$ be a bad line. Write
    $$
    \ell=\{p+t v:t\in\mathbb R\},
    \qquad
    p\in\mathbb R^2,\quad v=(v_1,v_2)\neq 0.
    $$
    Since $\ell$ is bad, the restriction of $h$ to $\ell$ is affine. Hence
    $$
    t\longmapsto h(p+t v)
    $$
    is a polynomial of degree at most $1$.

    On the other hand, the coefficient of $t^D$ in $h(p+t v)$ is precisely
    $$
    h_D(v).
    $$
    Indeed, the degree-$D$ contribution comes only from the highest homogeneous part $h_D$, and
    $$
    h_D(p+t v)=t^D h_D(v)+O(t^{D-1}).
    $$
    Since $D\ge 2$ and $h(p+t v)$ is affine, its $t^D$-coefficient must vanish. Therefore
    $$
    h_D(v)=0.
    $$

    Consequently, every bad-line direction $[v]\in\mathbb{RP}^1$ belongs to the projective zero set
    $$
    Z(h_D)
    =
    \{[v_1:v_2]\in\mathbb{RP}^1:
    h_D(v_1,v_2)=0\}.
    $$

    It remains to observe that this set is finite. If $v_1\neq 0$, write
    $$
    k=\frac{v_2}{v_1}.
    $$
    Then
    $$
    h_D(v_1,v_2)
    =
    v_1^D h_D(1,k),
    $$
    so such directions are zeros of the nonzero one-variable polynomial
    $$
    k\longmapsto h_D(1,k).
    $$
    There are at most $D$ such zeros, counting multiplicity. 
\end{proof}
}

We now pass from bad lines to a pointwise notion that is better suited to the multiscale decomposition. For a fixed point $\zeta_1\in [-1,1]^2$ in the parameter plane, we identify those points $\zeta_2\in [-1,1]^2$ that give rise to degenerate interactions with $\zeta_1$.

{ Recall that $l(\zeta_1,\zeta_2)$ denotes the line joining $\zeta_1$ and $\zeta_2$, and that $v(\zeta_1,\zeta_2)$ denotes the unit vector in the direction of the projection onto the parameter plane of the line
\[
T_{(\zeta_1,h(\zeta_1))}\mathcal S
\cap
T_{(\zeta_2,h(\zeta_2))}\mathcal S.
\]
}
\begin{definition}[Bad points for $\zeta_1$]\label{def badpoint}
Fix $\zeta_1\in [-1,1]^2$. We say that a point $\zeta_2\in [-1,1]^2$ is a \textbf{bad point associated with $\zeta_1$} if either of the following holds:
\begin{enumerate}
\item the normal vector to $\mathcal S$ at $\zeta_2$ coincides with the normal vector to $\mathcal S$ at $\zeta_1$; 
\item the tangent planes at $\zeta_1$ and $\zeta_2$ intersect in a line, and there exists a bad line passing through $\zeta_2$ whose corresponding straight line on $\mathcal S$ is parallel to this intersection line.
{ \item $\partial_{v(\zeta_{1},\zeta_{2})}^{2}h(\zeta_i)=0$ for any $i=1,2$.} 
\end{enumerate}
\end{definition}

We explain here why we focus on the quantity $\partial_{v(\zeta_{1},\zeta_{2})}^{2}h(\zeta_i)$. The quantity $\partial_{v(\zeta_1,\zeta_2)}^2 h(\zeta_i)$
is the second directional derivative of $h$ at $\zeta_i$
in the direction $v(\zeta_1,\zeta_2)$.
We denote by $\mathcal B_{\zeta_1}^1$ the set of points satisfying Definition \ref{def badpoint}(1) and by $\mathcal B_{\zeta_1}^2$ the set of points satisfying Definition \ref{def badpoint}(2). We shall prove that for $h$ of the form \eqref{eq reduction}, the only bad point of the first kind associated with $\zeta_1$ is $\zeta_1$ itself.

\begin{prop}\label{prop bad1}
For every fixed $\zeta_1\in[-1,1]^2$, one has
\[
\mathcal B_{\zeta_1}^1=\{\zeta_1\}.
\]
\end{prop}

This proposition follows from the fact that the hyperbolic paraboloid $h(\xi,\eta)=\xi\eta$ has injective Gauss map and that we are working with a sufficiently small perturbation of this surface. We include a proof for completeness.
\begin{proof}
By Definition \ref{def badpoint}(1), a point $\zeta_2\in[-1,1]^2$ belongs to $\mathcal B_{\zeta_1}^1$ if and only if the normal vector to $\mathcal S$ at $\zeta_2$ coincides with the normal vector to $\mathcal S$ at $\zeta_1$. Since $\mathcal S$ is the graph of $h$, this is equivalent to
\[
\nabla h(\zeta_2)=\nabla h(\zeta_1).
\]
Thus it suffices to prove that the gradient map
\[
G(\xi,\eta):=\nabla h(\xi,\eta)=(h_\xi(\xi,\eta),h_\eta(\xi,\eta))
\]
is injective on $[-1,1]^2$.

Recall that
\[
h(\xi,\eta)=\xi\eta+a_{2,0}\xi^2+a_{0,2}\eta^2+\sum_{i=3}^D\sum_{j=0}^i a_{i-j,j}\,\xi^{i-j}\eta^j,
\]
and that the coefficients satisfy the smallness condition \eqref{eq smallcoeff}. It follows uniformly on $[-1,1]^2$ that
\[
h_{\xi\eta}(\xi,\eta)=1+O(\varepsilon_0),\qquad
h_{\xi\xi}(\xi,\eta)=O(\varepsilon_0),\qquad
h_{\eta\eta}(\xi,\eta)=O(\varepsilon_0).
\]
Hence the Jacobian matrix of $G$, namely the Hessian matrix of $h$, is
\[
DG(\xi,\eta)=
\begin{pmatrix}
h_{\xi\xi}(\xi,\eta) & h_{\xi\eta}(\xi,\eta)\\
h_{\eta\xi}(\xi,\eta) & h_{\eta\eta}(\xi,\eta)
\end{pmatrix},
\]
and its determinant satisfies
\[
\det DG(\xi,\eta)
=
h_{\xi\xi}(\xi,\eta)h_{\eta\eta}(\xi,\eta)-h_{\xi\eta}(\xi,\eta)^2
=
-1+O(\varepsilon_0).
\]
If $\varepsilon_0>0$ is sufficiently small, then
\[
\det DG(\xi,\eta)\neq 0
\qquad\text{for all }(\xi,\eta)\in[-1,1]^2.
\]

Now suppose that $\zeta_2\in \mathcal B_{\zeta_1}^1$. Then
\[
G(\zeta_2)-G(\zeta_1)=0.
\]
Let $v:=\zeta_2-\zeta_1$. By the fundamental theorem of calculus,
\[
0
=
G(\zeta_2)-G(\zeta_1)
=
\left(\int_0^1 DG(\zeta_1+t(\zeta_2-\zeta_1))\,dt\right)(\zeta_2-\zeta_1).
\]
Set
\[
A:=\int_0^1 DG(\zeta_1+t v)\,dt.
\]
Then
\[
A=
\begin{pmatrix}
O(\varepsilon_0) & 1+O(\varepsilon_0)\\
1+O(\varepsilon_0) & O(\varepsilon_0)
\end{pmatrix},
\]
so
\[
\det A=-1+O(\varepsilon_0)\neq 0
\]
provided $\varepsilon_0$ is sufficiently small. Hence $A$ is invertible, and the identity
\[
Av=0
\]
implies $v=0$. Therefore $\zeta_2=\zeta_1$, and so
\[
\mathcal B_{\zeta_1}^1=\{\zeta_1\}.
\]
\end{proof}

Next, we show that if $\zeta_2$ is a bad point of the second kind associated with $\zeta_1$, then $\zeta_1$ and $\zeta_2$ lie on a common bad line. In particular, bad points of the second kind occur only along bad lines. We also prove that, under \eqref{eq smallcoeff}, every bad line is nearly parallel to one of the coordinate axes.

\begin{prop}\label{prop bad2}
\begin{enumerate}
    \item Every bad line makes angle $\lesssim \varepsilon_0$ with either the $\xi$-axis or the $\eta$-axis.
    \item Fix $\zeta_1\in[-1,1]^2$. Then every point of $\mathcal B_{\zeta_1}^2$ lies on a bad line passing through $\zeta_1$.
\end{enumerate}
\end{prop}

\begin{proof}

(1) Write
\[
h(\xi,\eta)=\xi\eta+\psi(\xi,\eta),
\]
where
\[
\psi(\xi,\eta)=a_{2,0}\xi^2+a_{0,2}\eta^2+\sum_{i=3}^D\sum_{j=0}^i a_{i-j,j}\xi^{i-j}\eta^j.
\]
By \eqref{eq smallcoeff}, all coefficients of $\psi$ are $O(\varepsilon_0)$.

Suppose the bad line $l$ in the original coordinates is
\[
\eta=k(\xi-\xi_0)+\eta_0.
\]
Since the angle between $l$ and the $\xi$-axis is less than $\frac{\pi}{4}$, we have $|k|<1$. Since $l$ is bad, the restriction of $h$ to $l$ is affine in $\xi$, and hence
\[
\frac{d^2}{d\xi^2}h\bigl(\xi,k(\xi-\xi_0)+\eta_0\bigr)\equiv 0.
\]
On the other hand,
\[
\frac{d^2}{d\xi^2}h\bigl(\xi,k(\xi-\xi_0)+\eta_0\bigr)
=
2k+\frac{d^2}{d\xi^2}\psi\bigl(\xi,k(\xi-\xi_0)+\eta_0\bigr).
\]
Because $|k|<1$ and all coefficients of $\psi$ are $O(\varepsilon_0)$, the polynomial
\[
\frac{d^2}{d\xi^2}\psi\bigl(\xi,k(\xi-\xi_0)+\eta_0\bigr)
\]
has coefficients $O(\varepsilon_0)$. Therefore
\[
2k+O(\varepsilon_0)\equiv 0,
\]
and hence
\[
k=O(\varepsilon_0).
\]
Thus, the rotation angle $\alpha=\arctan k$ also satisfies $\alpha=O(\varepsilon_0)$. The case where the bad line makes angle at most $\frac{\pi}{4}$ with the $\eta$-axis is treated identically. This proves part (1).

(2) It suffices to show that if $\zeta_2$ is a bad point of the second kind associated with $\zeta_1$, and if $l$ is the bad line passing through $\zeta_2$ and parallel to the projection onto the parameter plane of the intersection line of $T_{\zeta_1}\mathcal{S}$ and $T_{\zeta_2}\mathcal{S}$, then $\zeta_1$ must also lie on $l$.

We only treat the case where the angle between $l$ and the $\xi$-axis is less than $\frac{\pi}{4}$, since the other case is identical. After applying the affine change of variables $F_{\alpha,a}$ from Definition \ref{def badline}, we may assume that
\[
l=\{\eta=0\},
\]
and that
\[
g(\xi,\eta):=h\circ F_{\alpha,a}^{-1}(\xi,\eta)
\]
has the form
\[
g(\xi,\eta)=\eta\,\tilde h(\xi,\eta)+L(\xi),
\]
where $L$ is a linear function.

We claim that
\[
\partial_\xi \tilde h(\xi,\eta)=1+O(\varepsilon_0)
\]
uniformly on $[-1,1]^2$.

Since $F_{\alpha,a}$ is a translation followed by a rotation through an angle $O(\varepsilon_0)$, the transformed polynomial $g=h\circ F_{\alpha,a}^{-1}$ has the form
\[
g(\xi,\eta)=\xi\eta+r(\xi,\eta)+\ell(\xi,\eta)+c,
\]
where $\ell$ is due to a small rotation angle to compensate for the linear term of the preceding $\xi\eta$, $r$ is due to higher-order terms arising from the rotation of $\psi$, and $c$ is a constant. Every coefficient of $\ell$ and $r$ is $O(\varepsilon_0)$. Because $\eta=0$ is a bad line, the restriction $g(\xi,0)$ is linear in $\xi$. After subtracting this linear part, the remainder vanishes on $\{\eta=0\}$, hence is divisible by $\eta$. Therefore we may write
\[
g(\xi,\eta)-L(\xi)=\eta\bigl(\xi+\psi_1(\xi,\eta)\bigr),
\]
where $\psi_1$ is a polynomial all of whose coefficients are $O(\varepsilon_0)$. Thus
\[
\tilde h(\xi,\eta)=\xi+\psi_1(\xi,\eta),
\]
and so
\[
\partial_\xi \tilde h(\xi,\eta)=1+O(\varepsilon_0),
\]
as claimed.

Now write
\[
\zeta_1=(\xi_1,\eta_1),\qquad \zeta_2=(\xi_2,0).
\]
For the graph $z=g(\xi,\eta)$, the projection onto the parameter plane of the intersection line of the tangent planes at $\zeta_1$ and $\zeta_2$ has direction
\[
\bigl(g_\eta(\zeta_2)-g_\eta(\zeta_1),\, g_\xi(\zeta_1)-g_\xi(\zeta_2)\bigr).
\]
Since $\zeta_2$ is a bad point of the second kind associated with $\zeta_1$, this direction is parallel to the bad line $\eta=0$.
Hence its second component must vanish:
\[
g_\xi(\zeta_1)-g_\xi(\zeta_2)=0.
\]
Using
\[
g_\xi(\xi,\eta)=\eta\,\partial_\xi \tilde h(\xi,\eta)+L'(\xi),
\]
and the fact that $L'$ is constant, we obtain
\[
0=g_\xi(\zeta_1)-g_\xi(\zeta_2)
=\eta_1\,\partial_\xi \tilde h(\xi_1,\eta_1).
\]
Since $\partial_\xi \tilde h(\xi_1,\eta_1)=1+O(\varepsilon_0)$ and $\varepsilon_0$ is sufficiently small, this factor is nonzero. Therefore
\[
\eta_1=0.
\]
Hence $\zeta_1$ also lies on $\{\eta=0\}$, and therefore on the original bad line $l$. This proves the proposition.
\end{proof}

{This bad-line geometry is analogous in spirit to the exceptional-set analysis of Buschenhenke--Müller--Vargas \cite{BMV3}. In their work, the failure of strong transversality is described through the factorization of the hyperbolic transversality quantity, which leads to exceptional level curves and rectangular regions; see in particular \cite[Definition 3.2, Lemma 3.5, and Lemma 4.3]{BMV3}. Our formulation is different and more directly geometric: we isolate the actual affine lines in the parameter plane whose lifts are straight lines on the surface, and prove that such lines are nearly parallel to the coordinate axes and that bad points of the second kind lie on bad lines through the fixed base point.}

{
In fact, we prove a stronger result: the curvature of the restriction of $h$ to the line $l(p_1,p_2)$ controls the curvature at $p_1$ and $p_2$ in the direction of the tangent-plane intersection. Recall that $v(p_1,p_2)$ denotes the unit vector in the direction of the projection onto the parameter plane of the line
\[
T_{(p_1,h(p_1))}\mathcal S
\cap
T_{(p_2,h(p_2))}\mathcal S.
\]

\begin{lemma}[Curvature transfer]
\label{lem:curvature-transfer}
Let
$$
h(x,y)=xy+ax^2+by^2+r(x,y),
$$
where
$$
|a|+|b|\le\varepsilon_0,
\qquad
\|r\|_{C^3([-1,1]^2)}\le\Delta,
\qquad
0\le\Delta\ll \iota,
\qquad
0<\iota\le\varepsilon_0.
$$
For distinct $p_1,p_2\in[-1,1]^2$, write
$$
d=|p_2-p_1|,
\qquad
e=\frac{p_2-p_1}{d}.
$$
Assume that
$$
\left|
\partial_e^2h\bigl(p_1+t(p_2-p_1)\bigr)
\right|
\ge\iota,
\qquad 0\le t\le1.
$$
Then
$$
\left|\partial_{v(p_1,p_2)}^2h(p_i)\right|
\ge\frac{\iota}{2},
\qquad i=1,2.
$$
\end{lemma}

\begin{proof}
Retain the complete quadratic part and set
$$
A=
\begin{pmatrix}
2a&1\\
1&2b
\end{pmatrix},
\qquad
J=
\begin{pmatrix}
0&1\\
-1&0
\end{pmatrix}.
$$
Since $|a|+|b|\le\varepsilon_0$, for every unit vector $e$,
$$
|Ae|=1+O(\varepsilon_0),
\qquad
\det A=-1+O(\varepsilon_0^2).
$$
The bound on $r$ gives
$$
\nabla h(p_2)-\nabla h(p_1)
=dAe+\nabla r(p_2)-\nabla r(p_1)
=d\bigl(Ae+O(\Delta)\bigr).
$$
In particular, the gradient difference is nonzero. Rotating it
by $90^\circ$ and normalizing, we obtain
$$
v(p_1,p_2)
=
\frac{JAe}{|Ae|}+O(\Delta).
$$
Also,
$$
D^2h(p_i)=A+O(\Delta),
\qquad
\partial_e^2h(p_i)=e^{\mathsf T}Ae+O(\Delta).
$$

For every symmetric $2\times2$ matrix $A$, direct multiplication gives
$$
(JA)^{\mathsf T}A(JA)=\det(A)A.
$$
Consequently, for $i=1,2$,
$$
\begin{aligned}
\partial_{v(p_1,p_2)}^2h(p_i)
&=
\frac{(JAe)^{\mathsf T}A(JAe)}{|Ae|^2}+O(\Delta)\\
&=
\frac{\det A}{|Ae|^2}\,e^{\mathsf T}Ae+O(\Delta)\\
&=
\frac{\det A}{|Ae|^2}\,\partial_e^2h(p_i)+O(\Delta).
\end{aligned}
$$
All implicit constants are absolute. Since
$$
\frac{|\det A|}{|Ae|^2}\ge1-C\varepsilon_0,
$$
the assumed lower bound at the endpoints implies
$$
\begin{aligned}
\left|\partial_{v(p_1,p_2)}^2h(p_i)\right|
\ge(1-\varepsilon_0)\iota-\Delta \gtrsim\iota.
\end{aligned}
$$
\end{proof}

When we use this lemma later, we will first perform an isotropic rescaling so that the distance between the two points after scaling is $\sim 1$. The scaling constant is the distance between the two points. When the original distance between the two points is sufficiently small, the condition that $\|\cdot\|_{C^{3}}$ is sufficiently small is automatically satisfied.

Compared with Proposition \ref{prop bad2}, Lemma \ref{lem:curvature-transfer} assumes the stronger condition of a uniform curvature lower bound along the segment joining $p_1$ and $p_2$. In return, it gives a pointwise curvature lower bound at both $p_1$ and $p_2$ in the direction $v(p_1,p_2)$. 

We next bound the number of bad points of the third kind associated with a fixed $\zeta_1$. For fixed $\zeta_1\in[-1,1]^2$, define
$$
\mathcal B_{\zeta_1}^{3}
=
\left\{
\zeta_2\neq \zeta_1:
\partial_{v(\zeta_1,\zeta_2)}^2h(\zeta_i)=0,
\ i=1,2
\right\}.
$$

\begin{prop}
\label{prop:finite-third-kind}
For every fixed $\zeta_1\in[-1,1]^2$,
$$
\#\left(
\mathcal B_{\zeta_1}^{3}
\setminus
\mathcal B_{\zeta_1}^{2}
\right)
\lesssim_D 1.
$$
\end{prop}

\begin{proof}
Fix $p=\zeta_1$, write $H(z)=D^2h(z)$, and set
$$
V_p(z)
=
J\bigl(\nabla h(z)-\nabla h(p)\bigr),
\qquad
J=
\begin{pmatrix}
0&1\\
-1&0
\end{pmatrix}.
$$
By Proposition \ref{prop bad1}, $V_p(z)\neq0$ for $z\neq p$.

Let $e_+$ and $e_-$ be the two null directions of the indefinite
matrix $H(p)$. If $z\in\mathcal B_p^3$, then
$$
V_p(z)^TH(p)V_p(z)=0,
$$
so $V_p(z)\parallel e_\sigma$ for some
$\sigma\in\{+,-\}$.

Define
$$
F_\sigma(z)
=
\left\langle
\nabla h(z)-\nabla h(p),e_\sigma
\right\rangle,
\qquad
A_\sigma(z)
=
\left\langle
H(z)e_\sigma,e_\sigma
\right\rangle.
$$
Then
$$
\mathcal B_p^3
\subset
\bigcup_{\sigma=\pm}
\{F_\sigma=0,\ A_\sigma=0\}.
$$

Suppose a nontrivial curve $\gamma$ is contained in
$\{F_\sigma=A_\sigma=0\}$. Differentiating $F_\sigma(\gamma)=0$
gives
$$
\left\langle
H(\gamma)e_\sigma,\gamma'
\right\rangle=0.
$$
Together with $A_\sigma(\gamma)=0$ and the invertibility of
$H(\gamma)$, this implies
$$
\gamma'\parallel e_\sigma.
$$
Hence $\gamma$ lies in a line $L$ of direction $e_\sigma$.
Moreover, $F_\sigma=0$ implies that
$\partial_{e_\sigma}h$ is constant on an interval of $L$.
Since $h$ is a polynomial, $h|_L$ is affine, so $L$ is a bad line.
Also $V_p(z)\parallel e_\sigma$, hence every point of this curve
belongs to $\mathcal B_p^2$.

Therefore, after removing $\mathcal B_p^2$, the common zero sets
$\{F_\sigma=A_\sigma=0\}$ have only isolated points. Since
$$
\deg F_\sigma\leq D-1,
\qquad
\deg A_\sigma\leq D-2,
$$
Bézout's theorem yields at most $O_D(1)$ such points.
\end{proof}

}

 Moreover, we note that any two distinct points on a common bad line are bad points of the second kind associated with each other. Indeed, let $p_1,p_2\in[-1,1]^2$ be two distinct points lying on a bad line. By the lemma characterizing bad lines, the segment joining the corresponding points on $\mathcal S$ is a straight line segment contained in $\mathcal S$. Hence this straight line is contained in both tangent planes $T_{p_1}\mathcal S$ and $T_{p_2}\mathcal S$. Since these tangent planes are distinct, their intersection is exactly this line. Therefore $p_2$ is a bad point of the second kind associated with $p_1$, and symmetrically $p_1$ is a bad point of the second kind associated with $p_2$.

We now pass from the pointwise notion of badness to its square-scale analogue. This allows us to distinguish between pairs of frequency squares that behave transversely and those that do not. Let $\Delta>0$ be a scale parameter. For pairs of $\Delta$-squares at distance $\sim 1$, we introduce the notions of bad pairs and $\iota$-good pairs.

\begin{definition}[Bad pair]
For two $\Delta$-squares $\tau_1,\tau_2$ at distance $\sim 1$, we call $(\tau_1,\tau_2)$ a $\Delta$-bad pair if there exist points $\zeta_1\in\tau_1$ and $\zeta_2\in\tau_2$ such that $\zeta_2$ is a bad point associated with $\zeta_1$.
\end{definition}

{
\begin{definition}[$\iota$-good pair]\label{def stronggoodpair}
For two $\Delta$-squares $\tau_1,\tau_2$ at distance $\sim 1$, we call $(\tau_1,\tau_2)$ a $\iota$-good pair if for every $p_1\in\tau_1$ and $p_2\in\tau_2$,
\begin{equation}
\begin{aligned}
\operatorname{dist}(3\tau_1,3\tau_2)&\sim 1,\\
\inf_{\substack{\zeta_1\in3\tau_1\\\zeta_2\in3\tau_2}}
\min_{j=1,2}
\left|
\partial_{v(p_1,p_2)}^2h(\zeta_j)
\right|&\gtrsim \iota.
\end{aligned}
\end{equation}
\end{definition}
}

One important feature of our definitions is that good configurations are stable under small perturbations. 
{
\begin{lemma}\label{lem stability}
Let
$$
h(x,y)=xy+ax^2+by^2+r(x,y),
$$
where
$$
|a|+|b|\le\varepsilon_0,
\qquad
\|r\|_{C^3([-1,1]^2)}\le\Delta,
\qquad
0\le\Delta\ll \iota,
\qquad
0<\iota\le\varepsilon_0.
$$
Let $S$ be a $\iota$-good strip with $(\Delta,1)$-scales and let $\tau_{1},\tau_{2}\subset S$ be $\Delta$-squares with $\dist(\tau_{1},\tau_{2})\sim1$.
Then $(\tau_{1},\tau_{2})$ is a $\iota$-good pair.
\end{lemma}

\begin{proof}
The lemma follows from Lemma \ref{lem:curvature-transfer}, since the $\Delta$-scale perturbation does not affect the curvature lower bound.
\end{proof}

We will often rescale two points/squares to a distance of $\sim1$ later, and then repeatedly use the following lemma.
\begin{lemma}\label{lem perturbations}
    Let $p_{1},p_{2}$ be two points in parameter plane with distance $\sim1$. If $|\partial_{v_{(p_{1},p_{2})}}^{2}h(p_{i})|\gtrsim \iota$ and $|p_{1}-p_{1}'|+|p_{2}-p_{2}'|\lesssim \Delta$, where $\Delta \ll \iota$, then
    $$|\partial_{v_{(p_{1}',p_{2}')}}^{2}h(p_{i}')|\gtrsim \iota.$$
\end{lemma}
}

\begin{remark}
If $(\tau_1,\tau_2)$ is a $\iota$-good pair and $\tau_i'\subset \tau_i$ for $i=1,2$, then $(\tau_1',\tau_2')$ is also a $\iota$-good pair.
\end{remark}

We next introduce the notions of good strips and bad strips. Roughly speaking, a good strip is a rectangle whose long direction has a uniform curvature $\iota$, whereas a bad strip is one whose long direction is aligned with a bad line.

\begin{definition}[Bad strip]
Let $L$ be a strip of dimensions  $(\Delta, r)$, where $r\gtrsim \Delta$. We say that $L$ is a bad strip if there exists a bad line $l$ such that
\[
|l\cap L|\ge \frac r2.
\]
Equivalently, the central line of $L$ is a bad line up to an $O(\Delta)$ error.
\end{definition}

\begin{definition}[$\iota$-Good strip]
Let $L$ be a strip of dimensions $(\Delta, r)$, where $r\gtrsim \Delta$. We say that $L$ is a $\iota$-good strip if every line segment $l$ satisfying
\[
|l\cap L|\ge \frac r2
\]
is a $\iota$-good line segment.
\end{definition}

We shall also use the following thickened form of Lemma \ref{lem finite}. It follows from Lemma \ref{lem finite}, Proposition \ref{prop bad2}(1), and the bounded-overlap property of the fixed $\Delta$-strip discretization: bad lines have only $O_D(1)$ possible local incidences through each point, and every bad line is contained in $O_D(1)$ strips of the discretized family at scale $\Delta$.

\begin{lemma}
Let $\Delta$ be a dyadic number. Take a fixed $\Delta$-discretization in the parameter plane and let $\mathcal S_\Delta$ be the subcollection consisting of all $\Delta$-bad strips. Then the following properties hold.
\begin{enumerate}
\item For every bad line $\ell$, the set $N_\Delta(\ell)\cap[-1,1]^2$
is covered by $O_D(1)$ strips from $\mathcal S_\Delta$.

\item Every $\Delta$-square in a fixed $\Delta$-grid of $[-1,1]^2$
intersects at most $O_D(1)$ strips from $\mathcal S_\Delta$.

\end{enumerate}
\end{lemma}

We now recall a useful lemma that will be used repeatedly. Its proof is standard, relying only on the two-dimensional bilinear restriction theorem and Fubini's theorem.

\begin{lemma}[\cite{DW}]\label{lem bilinear restriction}
Let $\pi_1,\pi_2$ be two planes in $\mathbb R^3$ meeting at angle $\sim 1$, and let $\ell=\pi_1\cap\pi_2$. Suppose that $\widehat{F_j}$ is supported in the $\Delta$-neighborhood $N_\Delta(\pi_j)$ of $\pi_j$ for $j=1,2$. Partition each $N_\Delta(\pi_j)$ into rectangular boxes $b_j$ congruent to
\[
[-\Delta,\Delta]\times[-\Delta,\Delta]\times \mathbb R,
\]
with long axis parallel to $\ell$. Then
\[
\int_{\mathbb R^3}|F_1F_2|^2\sim \sum_{b_1,b_2}\int_{\mathbb R^3}|P_{b_1}F_1\,P_{b_2}F_2|^2,
\]
where $P_bF$ denotes the Fourier projection of $F$ to $b$.
\end{lemma}

\begin{remark}
The set $\mathcal B_{\zeta_1}^1$ consists precisely of those points $\zeta_2$ for which the tangent planes at $\zeta_1$ and $\zeta_2$ are parallel. In particular, for such pairs there is no intersection line of tangent planes, so Lemma \ref{lem bilinear restriction} does not apply.
\end{remark}

\section{$\iota$-bilinear refined decoupling}

In this section, we introduce the $\iota$-bilinear refined decoupling inequality for surfaces of the form \eqref{eq reduction}. 

{

\begin{definition}\label{def refineddecoupling}
Let $(\tau_1,\tau_2)$ be a $\iota$-good pair of arbitrary scale, and let $X\subset B_R$ be the union of a collection of pairwise disjoint $R^{1/2}$-balls $Q$. For $j=1,2$, let
\[
f_j=\sum_{T\in\mathbb T_j} f_T
\]
be a sum of scale-$R$ wave packets satisfying
\[
\supp(\widehat f_j)\subset \mathcal N_{R^{-1}}(\mathcal S_{\tau_j}),
\]
where $\mathcal S$ is a polynomial surface of the form \eqref{eq reduction}. Suppose that there exists $M_j\ge 1$ such that each $R^{1/2}$-ball $Q\subset X$ intersects at most $M_j$ tubes from $\mathbb T_j$.

We define $C(R;\iota)$ to be the smallest constant such that
\begin{equation}
\int_X |f_1f_2|^2 \le C(R;\iota)(M_1M_2)^{1/2}\prod_{j=1}^2 \Big(\sum_{T\in\mathbb T_j}\|f_T\|_4^4\Big)^{1/2}
\end{equation}
holds for all such configurations.
\end{definition}
}

For later use, set
\[
\overline C(S;\iota):=\sup_{1\le t\le S}C(t;\iota),\qquad S\ge1.
\]
Thus $\overline C(S;\iota)$ is non-decreasing in $S$. Also,
$C(S;\iota_2)\le C(S;\iota_1)$ whenever $0<\iota_1\le\iota_2$.

For the elliptic scenario, the refined decoupling inequality was presented in \cite{GIWO} and independently observed by Du-Zhang. For the hyperbolic scenario,
Demeter and Wu first developed the bilinear refined decoupling inequality in the model case $\mathbb{H}$.
We shall prove the following theorem, which is the bilinear refined decoupling inequality for surfaces of the form \eqref{eq reduction}.

{
\begin{theorem}\label{thm refineddecoupling}
For every sufficiently small $\varepsilon>0$, uniformly for
$\iota\ge R^{-\varepsilon^{9}}$, one has
\[
C(R;\iota)\lesssim_{\varepsilon,D}R^\varepsilon.
\]
\end{theorem}
}
Fix a sufficiently small target exponent $\varepsilon>0$ and put
$\epsilon=\varepsilon/10$. Fix a top scale $R_0\gg1$. By monotonicity
in the curvature parameter, it suffices to treat
$\iota=R_0^{-\varepsilon^{9}}$. Define
\[
K_1=R_0^{\epsilon^6},\qquad K_2=R_0^{\epsilon^4},\qquad
K_3=R_0^{\epsilon^2},\qquad K_0=\frac{K_2}{K_1}.
\]
These parameters and $\iota$ remain fixed throughout the argument.
Let $B\ge1$ be a fixed exponent in the crude wave-packet bound
$C(t;\iota)\lesssim_D t^B$, justified in the final proof below.
Choose a fixed $L\ge20$ and set
\[
R_*:=K_3^L=R_0^{L\epsilon^2}.
\]
We take $\varepsilon$ sufficiently small that $LB\epsilon^2\le\epsilon$.
In the local arguments below, $R$ denotes the current physical scale,
with $R_*<R\le R_0$; recursive constants at scales at most $R_*$
are handled by the crude bound. 
In particular,
\[
1\ll K_1\ll K_2\ll K_3\ll R_0^\epsilon,\qquad
K_2^{-1}\ll\iota,\qquad K_1/K_2\ll\iota^2.
\]

We next introduce the notion of generic position for pairs of frequency squares. Unlike the notions of good pair and transversality, it imposes a quantitative relation between the distance of two frequency squares and their scale.

{
\begin{definition}[$\iota$ generic position]
We say that a pair of $r$-squares is in \emph{generic position} if the following two conditions hold:
\begin{itemize}
\item the distance between their centers is at least $K_2r$;
\item the pair of $r$-squares satisfies the second property of $\iota$-good pair.
\end{itemize}
\end{definition}
}
The lower bound $d\ge K_2r$ ensures that after rescaling the distance between the two squares to $1$, the sidelength of each square becomes $O(K_2^{-1})$. In particular, the rescaled squares are small enough for the stability lemma to produce a good pair at the $\frac{1}{K}$ scale.

\begin{definition}
We say that an $r$-strip is in $\iota$-generic position if it is a $\iota$-good strip.
\end{definition}

\begin{remark}
If $(\tau_1,\tau_2)$ is in $\iota$-generic position and $\tau'_i\subset \tau_i$ for $i=1,2$, then $(\tau'_1,\tau'_2)$ is also in $\iota$-generic position.
\end{remark}
For convenience, we will abbreviate $\iota$-good line segment, $\iota$-good pair, and $\iota$-generic strip as good line, good pair, and generic strip, respectively.

Let $Q$ be an $R^{1/2}$-ball in $\mathbb R^3$. By rescaling and repeated applications of Lemma \ref{lem bilinear restriction}, we obtain the following lemma.

\begin{lemma}\label{lem equal}
Let $\tau_1,\tau_2$ be a pair of $r$-squares in generic position, and let
\[
d:=\dist(c(\tau_1),c(\tau_2))
\]
denote the distance between their centers. Assume that
\begin{equation}\label{eq c}
dR^{1/2}r\gtrsim K_2.
\end{equation}
Then for every $R^{1/2}$-ball $Q\subset \mathbb R^3$,
\[
\int_Q |f_{\tau_1}f_{\tau_2}|^2
\sim
\int_Q
\sum_{\omega_1\subset \tau_1}|f_{\omega_1}|^2
\sum_{\omega_2\subset \tau_2}|f_{\omega_2}|^2,
\]
{where $\omega_j$ are $(r/K_2,r)$-rectangles partitioning $\tau_j$ and in generic position.}
\end{lemma}

\begin{proof}
Write $c(\tau_j)=\zeta_j$ for the center of $\tau_j$. Since $\zeta_1\neq \zeta_2$, Proposition \ref{prop bad1} implies that the tangent planes $T_{\zeta_1}\mathcal S$ and $T_{\zeta_2}\mathcal S$ are not parallel. We therefore let
\[
v:=v(\zeta_1,\zeta_2)
\]
denote the projection onto the parameter plane of their intersection line.

 $(\tau_1,\tau_2)$ is in generic position, and this property is preserved under isotropic rescaling in the parameter plane. Hence, after translating the midpoint of $\zeta_1$ and $\zeta_2$ to the origin and rescaling by the factor $\frac{1}{d}$ in the parameter plane, the images $\widetilde\tau_1$ and $\widetilde\tau_2$ are at distance $\sim 1$ and form a good pair with scale $O\!\left(\frac{r}{d}\right)$ for the rescaled surface $\widetilde{\mathcal S}$ by Lemma \ref{lem stability}.

Moreover, after this rescaling, the tangent planes at the centers of $\widetilde\tau_1$ and $\widetilde\tau_2$ meet at angle $\sim 1$. Indeed, for the model surface $h(\xi,\eta)=\xi\eta$, the difference of the corresponding gradients is of size $\sim 1$ after normalization by $d$, and the perturbative assumption \eqref{eq smallcoeff} preserves this transversality when $\varepsilon_0$ is sufficiently small.

We now localize to the ball $Q$. Let $\phi_Q$ be a smooth bump function adapted to $Q$. Then
\[
\int_Q |f_{\tau_1}f_{\tau_2}|^2
\sim
\int_{\mathbb R^3} |\phi_Q f_{\tau_1}\,\phi_Q f_{\tau_2}|^2,
\]
and similarly for the reverse inequality. Since $Q$ has sidelength $R^{1/2}$, the Fourier transform $\widehat{\phi_Q}$ is supported at scale $R^{-1/2}$. Therefore
\[
\supp\bigl(\widehat{\phi_Q f_{\tau_j}}\bigr)
\subset N_{C R^{-1/2}}(\mathcal S_{\tau_j}).
\]
After the above rescaling by $\frac{1}{d}$ in the parameter plane, together with the corresponding parabolic rescaling in $\mathbb R^3$, the Fourier support of the rescaled function lies inside the $\frac{1}{R^{1/2}d^2}$-neighborhood of the rescaled surface $\widetilde{\mathcal S}_{\widetilde\tau_j}$. Moreover, $\widetilde{\mathcal S}$ is still given by a polynomial of the form \eqref{eq reduction}.

Since the rescaled squares $\widetilde{\tau}_j$ have sidelength $\frac{r}{d}$, the rescaled surface $\widetilde{\mathcal S}$ over each $\widetilde{\tau}_j$ differs from its tangent plane by $O\bigl((\frac{r}{d})^2\bigr)$. On the other hand, the Fourier support of $\widetilde f_{\widetilde\tau_j}$ lies in a $\frac{1}{R^{1/2}d^2}$-neighborhood of $\widetilde{\mathcal S}_{\widetilde\tau_j}$. Since
\[
\frac{1}{d^{2}R^{1/2}} \lesssim \frac{r}{dK_{2}}
\quad\text{and}\quad
\frac{r^{2}}{d^{2}}\lesssim \frac{r}{dK_{2}},
\]
by the assumptions $dR^{1/2}r\gtrsim K_{2}$ and $d\gtrsim rK_{2}$, it follows that the Fourier support of $\widetilde f_{\widetilde\tau_j}$ is contained in an $O\bigl(\frac{r}{dK_2}\bigr)$-neighborhood of the tangent plane to $\widetilde{\mathcal S}$ over $\widetilde\tau_j$.

Therefore we may apply Lemma \ref{lem bilinear restriction} in the rescaled coordinates with
\[
\Delta \sim \frac{r}{dK_{2}},
\]
using rectangular boxes whose long axis is parallel to the intersection line of the two tangent planes. We obtain
\[
\int_{\widetilde Q}
|\widetilde f_{\widetilde\tau_1}\widetilde f_{\widetilde\tau_2}|^2
\sim
\int_{\widetilde Q}
\sum_{\widetilde\omega_1\subset \widetilde\tau_1}
|\widetilde f_{\widetilde\omega_1}|^2
\sum_{\widetilde\omega_2\subset \widetilde\tau_2}
|\widetilde f_{\widetilde\omega_2}|^2.
\]
{ Here $\{\tilde \omega_i\}$ denotes the family of $\big(\frac{r}{dK_2},\frac{r}{d}\big)$-rectangles partitioning $\tilde \tau_i$, with long side parallel to the projection onto the parameter plane of the intersection line of the tangent planes at the centers $c(\tilde \tau_1)$ and $c(\tilde \tau_2)$. Since $(\tilde \tau_{1},\tilde \tau_{2})$ is a good pair, the curvature along the $v(\tilde\zeta_{1},\tilde \zeta_{2})$ direction has lower bounds $\iota$ within the $\frac{r}{d}\leq \frac{1}{K_{2}}\ll \iota$ scales. Thus, these $\tilde \omega_{j}$ are in generic position in rescaled parameter plane.}
Rescaling back, these correspond to rectangles $\omega_j\subset \tau_j$ of dimensions $\big(\frac{r}{K_2},r\big)$, and hence
\[
\int_Q |f_{\tau_1}f_{\tau_2}|^2
\sim
\int_Q
\sum_{\omega_1\subset \tau_1}|f_{\omega_1}|^2
\sum_{\omega_2\subset \tau_2}|f_{\omega_2}|^2.
\]
{ Since the property for generic position is preserved under isotropic rescaling in the parameter plane, these $\omega_{j}$ are in generic position.}
This completes the proof.
\end{proof}
Note that the reverse inequality in Lemma \ref{lem equal}, namely that the right-hand side is bounded by the left-hand side, only uses the condition
$
drR^{1/2}\gtrsim K_2.
$

For a rectangle $\omega_i$, define
\[
g_{\omega_i}:=\Biggl(\sum_{s_i\not\sim s_i'\subset \omega_i}|f_{s_i}f_{s_i'}|\Biggr)^{1/2}.
\]
Here we write $s_i\not\sim s_i'$ to mean that $s_i$ and $s_i'$ are \emph{not adjacent}, in the sense that the distance between their centers is at least twice the length of their long side.

We now introduce a key lemma. It shows that estimating $\int_Q |f_{\tau_1}f_{\tau_2}|^2$ can be reduced either to estimating an expression of the same form at a smaller frequency scale, or to estimating a term involving $g_{\omega_1}$ and $g_{\omega_2}$.

\begin{lemma}\label{lem dichotomy}
Let $(\tau_1,\tau_2)$ be a pair of $r$-squares in generic position, and let $d$ denote the distance between their centers.
Assume that
\[
drR^{1/2}\ge K_2.
\]
Then
\begin{equation}\label{eq dich}
\int_Q |f_{\tau_1}f_{\tau_2}|^2
\lesssim
\int_Q \sum_{\beta_1\subset \tau_1}|f_{\beta_1}|^2 \sum_{\beta_2\subset \tau_2}|f_{\beta_2}|^2
+K_1^{O(1)}\sum_{\omega_1\subset \tau_1}\sum_{\omega_2\subset \tau_2}\int_Q |g_{\omega_1}g_{\omega_2}|^2.
\end{equation}

{ Here $\{\beta_i\}$ denotes the family of $\frac{r}{K_1}$-squares partitioning $\tau_i$, and $\{\omega_i\}$ denotes the family of $\big(\frac{r}{K_2},r\big)$-rectangles partitioning $\tau_i$ and each rectangle from $\{\omega_i\}$ is in generic position.}
\end{lemma}
\begin{proof}

We apply Lemma \ref{lem equal} at the $\big(\frac{r}{K_{2}},r\big)$-scale:
\begin{equation}
\begin{aligned}
    &\int_{Q}|f_{\tau_{1}}f_{\tau_{2}}|^{2}\lesssim \sum_{\omega_{1}\subset \tau_{1},\,\omega_{2}\subset \tau_{2}} \int_{Q} |f_{\omega_{1}}|^{2}|f_{\omega_{2}}|^{2}\\
  &\lesssim \sum_{\omega_{1}\subset \tau_{1},\,\omega_{2}\subset \tau_{2}} \Big(\sum_{s_{1}\subset \omega_{1},\,s_{2}\subset \omega_{2}} \int_{Q} |f_{s_{1}}|^{2}|f_{s_{2}}|^{2}+K_{1}^{O(1)}\max_{s_{1}\not\sim s_{1}'\subset \omega_{1}}\max_{s_{2}\not\sim s_{2}'\subset \omega_{2}}\int_{Q}|f_{s_{1}}f_{s_{1}'}f_{s_{2}}f_{s_{2}'}|\Big)\\
  & \lesssim \sum_{s_{1}\subset \tau_{1},\,s_{2}\subset \tau_{2}}\int_{Q} |f_{s_{1}}|^{2}|f_{s_{2}}|^{2}+K_{1}^{O(1)}\sum_{\omega_{1}\subset \tau_{1}}\sum_{\omega_{2}\subset \tau_{2}}\int_{Q}|g_{\omega_{1}}g_{\omega_{2}}|^{2}\\
  &=\sum_{\beta_{1}\subset \tau_{1},\,\beta_{2}\subset \tau_{2}}\sum_{s_{1}\subset \beta_{1},\,s_{2}\subset \beta_{2}}\int_{Q}|f_{s_{1}}|^{2}|f_{s_{2}}|^{2}+K_{1}^{O(1)}\sum_{\omega_{1}\subset \tau_{1}}\sum_{\omega_{2}\subset \tau_{2}}\int_{Q}|g_{\omega_{1}}g_{\omega_{2}}|^{2}.
\end{aligned}
\end{equation}

Since
\[
d\cdot \frac{r}{K_{1}}\cdot R^{1/2} \ge \frac{K_{2}}{K_{1}},
\]
we may apply the reverse inequality in Lemma \ref{lem equal} at the $\big(\frac{r}{K_{2}},\frac{r}{K_{1}}\big)$-scale, with $(r,K_2)$ replaced by $\big(\frac{r}{K_1},\frac{K_2}{K_1}\big)$, to obtain
\begin{equation}
    \int_{Q}|f_{\beta_{1}}|^{2}|f_{\beta_{2}}|^{2}\sim\sum_{s_{1}\subset \beta_{1},\,s_{2}\subset \beta_{2}}\int_{Q}|f_{s_{1}}|^{2}|f_{s_{2}}|^{2}.
\end{equation}

Combining the above estimates, we obtain \eqref{eq dich}.
    
\end{proof}
This decomposition naturally leads to a dichotomy according to which of the two terms in \eqref{eq dich} dominates.

\begin{definition}\label{def NB}
Let $(\tau_{1},\tau_{2})$ be a pair of $r$-squares in generic position, and let $d$ denote the distance between their centers. Assume that
\[
drR^{1/2}\ge K_{2}.
\]
We call $(\tau_{1},\tau_{2})$ \textbf{narrow} relative to $Q$ if the first term on the right-hand side of \eqref{eq dich} dominates, and \textbf{broad} relative to $Q$ otherwise.
\end{definition}

{
The following lemma requires that the strip has a lower bound $\iota$ on its curvature along the long direction. 
}

\begin{lemma}[\cite{DW}]\label{lem 3.6}
Let $s_{i},s_{i}'$ be {non-adjacent} rectangles of dimensions $\big(\frac{r}{K_{2}},\frac{r}{K_{1}}\big)$ contained in some $\big(\frac{r}{K_{2}},r\big)$-rectangle $\omega_{i}$ in generic position. Assume that
\[
r^{2}R^{1/2}\ge K_{1}K_{2}.
\]
Then
\begin{equation}
\int_{Q}|f_{s_{i}}f_{s_{i}'}|^{2}\lesssim \iota^{-O(1)}\int_{Q} \sum_{t_{i}\subset s_{i}}|f_{t_{i}}|^{2}\sum_{t_{i}'\subset s_{i}'}|f_{t_{i}'}|^{2},
\end{equation}
where $t_{i}$ and $t_{i}'$ are $\frac{r}{K_{2}}$-squares partitioning $s_{i}$ and $s_{i}'$, respectively.
\end{lemma}

In Demeter--Wu \cite{DW}, the corresponding lemma is stated for the almost-adjacent case, which is the worst case among the non-adjacent configurations. The same bound holds for rectangles that are farther apart. For simplicity, we state the lemma for general non-adjacent rectangles and we still deal with the same worst case.

By rescaling and Definition \ref{def refineddecoupling}, we obtain the following proposition.

\begin{prop}\label{prop decouplingconstant}
Assume that $rR^{1/2}\ge K_1$ and $r\lesssim \frac{1}{K_{2}}$. Let $\omega$ be a good strip of dimensions $\big(\frac{r}{K_2},r\big)$, and let $s,s'$ be two $\big(\frac{r}{K_2},\frac{r}{K_1}\big)$-rectangles inside $\omega$ that are {not adjacent.} Suppose that $X$ is a union of $R^{1/2}$-balls such that each ball intersects at most $M$ tubes from $\mathbb{T}(s)$ and at most $M'$ tubes from $\mathbb{T}(s')$. Then
\[
\int_X |f_s f_{s'}|^2
\lesssim
\iota^{-O(1)}\overline C\bigg(Rr^{2};\iota\bigg)(MM')^{1/2}
\Bigg(\sum_{T\in \mathbb{T}(s)} \|f_T\|_4^4\Bigg)^{1/2}
\Bigg(\sum_{T\in \mathbb{T}(s')} \|f_T\|_4^4\Bigg)^{1/2}.
\]
\end{prop}
\begin{proof}
Since $\omega$ is a good strip and $s,s'$ are {not adjacent}, every line joining a point of $s$ to a point of $s'$ has a uniform curvature after the surface map. 

We assume the distance of $s,s'$ is $\frac{r}{K_{1}}$ and this is the worst case. We translate the midpoints of these two rectangles to the origin, perform the isotropic scaling
\[
\xi=\frac{r}{K_{1}}\tilde{\xi}, \qquad \eta=\frac{r}{K_{1}}\tilde{\eta},
\]
on the parameter plane, and apply the parabolic rescaling
\[
\gamma=\Bigl(\frac{r}{K_{1}}\Bigr)^{2}\tilde{\gamma},
\]
in the third direction, where $\tilde{\xi},\tilde{\eta},\tilde{\gamma}$ are the new coordinates. Since $r\lesssim K_2^{-1}$, the rescaled surface satisfies the hypotheses of Lemma \ref{lem:curvature-transfer}.

We denote by $\tilde{s},\tilde{s}'$ the resulting rectangles with dimensions $(\frac{K_{1}}{K_{2}},1)$. The rescaled surface $\tilde{\mathcal S}$ is given by
\[
\tilde h=\tilde{\xi}\tilde{\eta}+\frac{K_{1}^{2}}{r^{2}}\psi\Bigl(\frac{r}{K_{1}}\tilde{\xi},\frac{r}{K_{1}}\tilde{\eta}\Bigr),
\]
modulo harmless affine terms. It is straightforward to verify that $\tilde{\mathcal S}$ still has the form \eqref{eq reduction}. Moreover, the original support in the $\frac{1}{R}$-neighborhood of $\mathcal S$ is carried to the $\frac{K_{1}^{2}}{Rr^{2}}$-neighborhood of $\tilde{\mathcal S}$, and the multiplicities $M,M'$ are preserved under this rescaling.
{ Furthermore, one checks that $\partial^{2}_{v(\zeta_{1},\zeta_{2})}h(\zeta_{i})\gtrsim \iota$ for any $\zeta_{i}\in \tilde s_{i}$. Fix $c=1/2$ and let $K=\iota^{-2}$. Then $\frac{K_{1}}{K_{2}}\ll\frac{1}{K}\ll\iota$.
We partition each rectangle $\tilde s_i$ into rectangles of dimensions $(K_1/K_2,1/K)$. Each such rectangle is contained in a $1/K$-square, which, with a slight abuse of notation, we also denote by $\tilde t_i$. Since $\iota=K^{-1/2}$ and $\frac{1}{K}\ll \iota$, we obtain that each $(\tilde t_{1},\tilde t_{2})$ is a $\iota$-good pair by Lemma \ref{lem perturbations}.
}

Therefore, by Definition \ref{def refineddecoupling}, summing up and rescaling back, we obtain
\[
\int_{X}|f_{s}f_{s'}|^{2}\lesssim \iota^{-O(1)} \overline C\Bigl(\frac{Rr^{2}}{K_{1}^{2}};\iota\Bigr)(MM')^{1/2}
\Bigl(\sum_{T\in \mathbb{T}(s)}\|f_{T}\|_{4}^{4}\Bigr)^{1/2}
\Bigl(\sum_{T\in \mathbb{T}(s')}\|f_{T}\|_{4}^{4}\Bigr)^{1/2}.
\]

For a general separation $d$ between $s$ and $s'$, the same argument gives
$$
\int_{X}|f_{s}f_{s'}|^{2}\lesssim \iota^{-O(1)}\overline C\Bigl(Rd^{2};\iota\Bigr)(MM')^{1/2}
\Bigl(\sum_{T\in \mathbb{T}(s)}\|f_{T}\|_{4}^{4}\Bigr)^{1/2}
\Bigl(\sum_{T\in \mathbb{T}(s')}\|f_{T}\|_{4}^{4}\Bigr)^{1/2}.
$$
By the definition of $\overline C$ and $d\lesssim r$, we have 
$$
\int_{X}|f_{s}f_{s'}|^{2}\lesssim\iota^{-O(1)}\overline C\Bigl(Rr^{2};\iota \Bigr)(MM')^{1/2}
\Bigl(\sum_{T\in \mathbb{T}(s)}\|f_{T}\|_{4}^{4}\Bigr)^{1/2}
\Bigl(\sum_{T\in \mathbb{T}(s')}\|f_{T}\|_{4}^{4}\Bigr)^{1/2}.
$$
\end{proof}

We now prove the key technical lemma, namely a dichotomy for the quantities $g_\omega$. Either one obtains a similar estimate at a smaller frequency scale, preserving the similar structure, or one arrives directly at the desired estimate. 

\begin{prop}\label{prop NB}
Assume that $\frac{1}{K_{3}}\lesssim r\lesssim \frac{1}{K_2}$. Fix an $R^{1/2}$-ball $Q$. Let $\{\omega\}$ be a collection of pairwise disjoint $\big(\frac{r}{K_2},r\big)$-rectangles such that every pair is in  generic position.
Then one of the following alternatives holds:
\begin{enumerate}
\item There exists $r'$ with
\[
R^{-1/2}K_1 \lesssim r' \le \frac{r}{K_2},
\]
and a collection of pairwise disjoint $\big(\frac{r'}{K_2},r'\big)$-rectangles $\omega' \subset \bigcup \omega$, again in generic position, such that
\begin{equation}\label{eq 3.81}
\sum_{\omega} \| g_{\omega} \|_{L^{4}(Q)}^{2}
\lesssim_{\epsilon} \iota^{-O(1)}
\left( \frac{r}{r'} \right)^{100 \frac{\log K_{1}}{\log K_{2}}}
\sum_{\omega'} \| g_{\omega'} \|_{L^{4}(Q)}^{2}.
\end{equation}

\item We have
\begin{equation}\label{eq 3.82}
\sum_{\omega} \| g_{\omega} \|_{L^{4}(Q)}^{2}
\lesssim_{\epsilon}
R_0^{\epsilon/4} \sum_{\theta \subset \bigcup \omega} \| f_{\theta} \|_{L^{4}(Q)}^{2},
\end{equation}
where the sum is taken over a partition of $\bigcup \omega$ into $R^{-1/2}$-squares $\theta$.
\end{enumerate}
\end{prop}

\begin{proof}
For any $\omega$, we can choose $\frac{r}{K_{1}}$-separated $\bigl(\frac{r}{K_{2}},\frac{r}{K_{1}}\bigr)$-rectangles $s_{1}(\omega), s_{2}(\omega) \subset \omega$ such that
\[
\|g_{\omega}\|_{L^{4}(Q)}^{2} \leq K_{1}^{O(1)} \left( \int_{Q} |f_{s_{1}(\omega)} f_{s_{2}(\omega)}|^{2} \right)^{\frac{1}{2}}.
\]

Since $r \gtrsim \frac{1}{K_{3}}$, we have $r^{2}R^{\frac{1}{2}} \gtrsim K_{1}K_{2}$. Since each $\omega$ is in generic position, applying Lemma \ref{lem 3.6} to each $\omega$ gives
\begin{equation}\label{eq partition1}
    \sum_{\omega} \|g_{\omega}\|_{L^{4}(Q)}^{2} \lesssim K_{1}^{O(1)} \sum_{\omega} \left( \iota^{-O(1)}\sum_{\alpha_{1}' \subset s_{1}(\omega)} \sum_{\alpha_{2}' \subset s_{2}(\omega)} \int_{Q} |f_{\alpha_{1}'} f_{\alpha_{2}'}|^{2} \right)^{\frac{1}{2}},
\end{equation}
where the pairs $(\alpha_{1}', \alpha_{2}')$ consist of two $\frac{r}{K_{2}}$-squares whose distance $d'$ is at least $\frac{r}{K_{1}}$ inside $\omega$. 

We partition $\alpha_{1}'$ and $\alpha_{2}'$ into $\frac{r}{K_{1}K_{2}}$-squares $\alpha_{1}$ and $\alpha_{2}$. There are $K_{1}^{2}$ squares $\alpha_{i}$ such that $\alpha_{i}\subset \alpha_{i}'$. By Cauchy-Schwarz, we have
\begin{equation}\label{eq cs}
\int_{Q} |f_{\alpha_{1}'} f_{\alpha_{2}'}|^{2}\lesssim K_{1}^{4}\sum_{\alpha_{1}\subset\alpha_{1}'}\sum_{\alpha_{2}\subset\alpha_{2}'}\int_{Q} |f_{\alpha_{1}} f_{\alpha_{2}}|^{2}.    
\end{equation}

We rescale on the parameter plane so that the distance of $\alpha_{1},\alpha_{2}$ is $\sim1$. Let $d$ denote their original distance. Since $d\lesssim r\lesssim \frac{1}{K_{2}}\ll \iota$, the condition of Lemma \ref{lem stability} is satisfied.
Since $\omega$ is in generic position, and the angle between the long side of $\omega$ and the line joining $\alpha_{1}$ and $\alpha_{2}$ is less than $\frac{K_{1}}{K_{2}}$, 
Lemma \ref{lem stability} implies that $(\alpha_{1},\alpha_{2})$ is a good pair. Actually, after rescaling the parameter plane, the resulting squares $\tilde{\alpha}_{1}, \tilde{\alpha}_{2}$ have distance $1$, and only then do they form a good pair in the strict sense. However, an isotropic rescaling does not change the second curvature condition of a good pair. Here, for convenience, we call $(\alpha_{1}, \alpha_{2})$ a good pair.
Moreover, we have
\[
d' \gtrsim \frac{r}{K_{1}} = K_{2} \cdot \frac{r}{K_{1}K_{2}},
\]
so each pair $(\alpha_{1},\alpha_{2})$ is in generic position.
Furthermore, the sidelength $\frac{r}{K_{1}K_{2}}$ of the squares and the distance $d'$ between them satisfy
\[
d' \cdot \frac{r}{K_{1}K_{2}} R^{\frac{1}{2}} \geq K_{2},
\]
since $r \gtrsim \frac{1}{K_{3}}$. Therefore, by Definition \ref{def NB}, each pair is either narrow or broad. We may assume that all pairs $(\alpha_{1},\alpha_{2})$ are of the same type, either narrow or broad. For a detailed explanation, see \cite{DW}.

{\it Case 1}: Suppose all pairs $(\alpha_{1},\alpha_{2})$ are narrow.
We start with each $\frac{r}{K_{1}K_{2}}$-pair $(\alpha_{1},\alpha_{2})$ and obtain
\[
\int_{Q}|f_{\alpha_{1}}f_{\alpha_{2}}|^{2}\leq C\int_{Q}\sum_{\beta_{1}\subset \alpha_{1}}|f_{\beta_{1}}|^{2}\sum_{\beta_{2}\subset \alpha_{2}}|f_{\beta_{2}}|^{2},
\]
where $\beta_{i}$ are $\frac{r}{K_{1}^{2}K_{2}}$-squares partitioning $\alpha_{i}$.

By the same reasoning, we can verify that $(\beta_{1}, \beta_{2})$ satisfies Definition \ref{def NB}; therefore, a dichotomy exists for $(\beta_{1}, \beta_{2})$.
Let us assume that the narrow case continues for as many steps $m$ as possible, where $m\ge 1$. We obtain
\begin{equation}\label{eq partition2}
    \int_{Q}|f_{\alpha_{1}}f_{\alpha_{2}}|^{2}\leq C^{m}\int_{Q}\sum_{\sigma_{1}\subset \alpha_{1}}|f_{\sigma_{1}}|^{2}\sum_{\sigma_{2}\subset\alpha_{2}}|f_{\sigma_{2}}|^{2},
\end{equation}
where $\sigma_{i}$ are $r_{1}\coloneqq \frac{r}{K_{2}K_{1}^{m+1}}$-squares partitioning $\alpha_{1}$. For all $(\alpha_{1},\alpha_{2})$, the value of $m$ is the same. Then we have the following two subcases.

\textbf{Subcase 1:} Suppose $R^{-\frac{1}{2}}\lesssim r_{1}\lesssim {K_{3}K_{2}K_{1}}R^{-\frac{1}{2}}$. In this subcase, subdivide into $R^{-1/2}$-squares and apply Minkowski's inequality. The remaining subdivision and the preceding estimates cost at most
\[
A_{\epsilon,D}\iota^{-O_D(1)}(K_1K_2K_3)^{O_D(1)}
\le A_{\epsilon,D}R_0^{\epsilon/4}
\]
for sufficiently small $\varepsilon$, since $\iota=R_0^{-\varepsilon^{9}}$ and $\epsilon=\varepsilon/10$. This gives \eqref{eq 3.82}.

\textbf{Subcase 2:} Suppose $r_{1}\gtrsim K_{1}K_{2}K_{3}R^{-\frac{1}{2}}$. 
We denote the distance between $\sigma_{1}$ and $\sigma_{2}$ by $d_{1}$. 
We note that the distance between squares essentially does not decrease when we scale down at each step. 
Thus, $d_{1}\gtrsim \frac{r}{K_{1}}\ge K_{2}r_{1}$ and $(\sigma_{1},\sigma_{2})$ is in a generic position. 
Moreover, since $r \gtrsim \frac{1}{K_{3}}$ and $r_{1} \gtrsim K_{1}K_{2}K_{3}R^{-\frac{1}{2}}$, it follows that
\[
d_{1} r_{1} R^{\frac{1}{2}} \gtrsim K_{2}.
\] 
Therefore, $(\sigma_{1},\sigma_{2})$ is narrow or broad according to Definition \ref{def NB}. 
Since the narrow case will halt after $m$ steps, we obtain 
\begin{equation}\label{eq partition3}
    \int_{Q}|f_{\sigma_{1}}f_{\sigma_{2}}|^{2}\lesssim K_{1}^{O(1)}\sum_{\omega_{1}\subset \sigma_{1}}\sum_{\omega_{2}\subset \sigma_{2}}\int_{Q}|g_{\omega_{1}}g_{\omega_{2}}|^{2}.
\end{equation}
Here $\omega_{i}$ are $(\frac{r_{1}}{K_{2}},r_{1})$-rectangles in generic position.

Let us explain how the orientation of $\omega_{i}$ is determined: it is determined by the pair $(\sigma_{1},\sigma_{2})$. Before applying Lemma \ref{lem bilinear restriction}, we do isotropic rescaling so that the distance between $\tilde \sigma_{1}$ and $\tilde \sigma_{2}$ is $\sim 1$, where $\tilde \sigma_{j}$ denotes the square $\sigma_{j}$ after rescaling and use $\bar h$ to denote the surface map after rescaling. We then use Lemma \ref{lem bilinear restriction} and partition $\tilde \sigma_{j}$ into $\tilde \omega_{j}'$ along the projection of the line of intersection of the two tangent planes. Finally, we rescale $\tilde \omega_{j}'$ back to obtain $\omega_{j}'$. By Lemma \ref{lem stability}, we note that $\tilde \omega_{j}'$ are good strips in rescaled parameter plane. Since $\tilde \omega_{j}'$ are good strips in rescaled parameter plane, $\omega_{j}'$ are still good strip after isotropic transformation. 

Notice that for different pairs $(\sigma_{1},\sigma_{2})$, the orientation may be different. However, we can rearrange the orientations of $\omega_{j}$ to be uniform. For any two points $p_{1}\in \tilde \sigma_{1}$ and $p_{2}\in \tilde \sigma_{2}$, the intersecting line of their tangent planes points in $(\bar h_{\eta}(p_{2})-\bar h_{\eta}(p_{1}),\bar h_{\xi}(p_{1})-\bar h_{\xi}(p_{2}))$ and the original strip $\omega$ after rescaling has width no more than $\frac{K_{1}}{K_{2}}$, then the angular deviation after rearrangement does not exceed $\frac{K_{1}}{K_{2}}$.
Moreover, since isotropic dilation preserves angles and the eccentricity of $\omega_{j}'$ itself is $\frac{1}{K_{2}}$, we incur an acceptable loss of $K_{1}^{O(1)}$.

{ Furthermore, every $\omega_{j}'$ is in generic position. Indeed, $\omega$ is in generic position and the curvature of $h(\omega)$ has a uniform bound $\iota$ along the long direction. Then, we perform an isotropic rescaling with scaling factor $\frac{r}{K_{1}}$. Since $r\lesssim \frac{1}{K_{2}}$, we have
$$\sum_{i+j\ge 3}\|\partial_{\xi}^{i}\partial_{\eta}^{j}\bar h\|_{\infty}\lesssim \frac{1}{K_{2}}\ll \iota.$$
Since the sidelength of $\tilde \sigma_{j}$ is at most $\frac{1}{K_{2}}$, the $(\tilde\sigma_{1},\tilde \sigma_{2})$ is a $\iota$-good pair by Lemma \ref{lem stability} and Lemma \ref{lem perturbations}. 
Thus, $\bar h(\tilde \omega_{j}')$ has a uniform curvature $\iota$ along the long direction according to Lemma \ref{lem:curvature-transfer} and stability analysis. Rescaling back, we conclude that each $\omega_{j}$ is in generic position.  
}

Combining with \eqref{eq partition1}\eqref{eq cs} \eqref{eq partition2} and \eqref{eq partition3}, we obtain
\begin{equation}
\begin{aligned}
    &\|g_{\omega}\|_{L^{4}(Q)}^{2}\lesssim C^{m}K_{1}^{O(1)}(\iota^{-O(1)}\sum_{\omega_{1}\subset s_{1}(\omega)}\sum_{\omega_{2}\subset s_{2}(\omega)}\int_{Q}|g_{\omega_{1}}g_{\omega_{2}}|^{2})^{\frac{1}{2}}\\ &\lesssim_{\epsilon}\iota^{-O(1)}K_{1}^{100}(\int_{Q}(\sum_{\omega'\subset \omega}|g_{\omega'}|^{2})^{2})^{\frac{1}{2}}. 
\end{aligned}
\end{equation}
Here $\omega'$ is simply the generic notation for either $\omega_{1}$ or $\omega_{2}$ and a rectangle with dimensions $(\frac{r_{1}}{K_{2}},r_{1})$.
Next, sum over $\omega$ and apply the Minkowski inequality, then we have
\begin{equation}
    \sum_{\omega}\|g_{\omega}\|_{L^{4}(Q)}^{2}\lesssim_{\epsilon}\iota^{-O(1)} K_{1}^{100} \sum_{\omega}\sum_{\omega'\subset \omega}\|g_{\omega'}\|_{L^{4}(Q)}^{2}.
\end{equation}

Finally, we take $r'=r_1$. Since $r_1=r/(K_2K_1^{m+1})$ for some $m\ge 1$, we have
\[
\Bigl(\frac{r}{r'}\Bigr)^{100\frac{\log K_1}{\log K_2}}
\ge K_1^{100},
\]
which absorbs the factor $K_{1}^{100}$. Hence we obtain \eqref{eq 3.81}.

{\it Case 2:} Suppose all pairs $(\alpha_{1},\alpha_{2})$ are broad. This is the case $m = 0$ in \textbf{Subcase 2}. By performing the same discussion as in \textbf{Subcase 2}, we obtain the inequality \eqref{eq 3.81}.

\end{proof}

\begin{remark}
 For different values of $Q$, we obtain different partitions $\{\omega'\}$. However, we note that each partition is completely determined by the scale $r'$ of $\omega'$ because
 we fix the direction of $\omega'$ with an acceptable loss. Since the scale $r'$ can take $m$ distinct values, there are $m$ types of partitions, where $m = O_{\epsilon}(1)$.
\end{remark}

Repeating the previous proposition until the strip scale reaches $K_3^{-1}$ gives the following version in the range needed below.

\begin{prop}\label{prop downscale}
Let $K_1R^{-1/2}\lesssim r\lesssim K_2^{-1}$. Fix an $R^{\frac{1}{2}}$-ball $Q$. Assume $\{\omega\}$ is a collection of pairwise disjoint $(\frac{r}{K_{2}},r)$-rectangles in generic position. Then there is a scale $K_{1}R^{-\frac{1}{2}}\lesssim r' \lesssim K_{3}^{-1}$ and there is a collection $\{\omega'\}$ consisting of pairwise disjoint $(\frac{r'}{K_{2}},r')$-rectangles $\omega'\subset \bigcup\omega$ in  generic position such that
\begin{equation}
    \sum_{\omega}\|g_{\omega}\|_{L^{4}(Q)}^{2}\lesssim_{\epsilon}\iota^{-B_D/\epsilon^{2}}\sum_{\omega'}(\frac{r}{r'})^{100\frac{\log K_{1}}{\log K_{2}}}\|g_{\omega'}\|_{L^{4}(Q)}^{2}+R_0^{\epsilon/2}\sum_{\theta\subset \cup \omega} \|f_{\theta}\|_{L^{4}(Q)}^{2},
\end{equation}
where $B_D>0$ is a sufficiently large fixed constant depending only on $D$, and the last sum is over a partition of $\bigcup \omega$ into $R^{-\frac{1}{2}}$-squares $\theta$. 
\end{prop}

\begin{proof}
If $r\le K_3^{-1}$, take $r'=r$ and $\{\omega'\}=\{\omega\}$.
Otherwise, repeat Proposition \ref{prop NB} until the strip scale is
at most $K_3^{-1}$ or the terminal alternative occurs. There are at
most $O(1+\log K_3/\log K_2)=O(\epsilon^{-2})$ strip iterations.
The product of the curvature
losses is bounded by $\iota^{-B_D/\epsilon^2}$ after increasing $B_D$.

If a branch terminates, its accumulated loss is at most
$R_0^{50\epsilon^2}$, since all intermediate strip scales are at least
$R^{-1/2}$. Consequently, the additional cost on such a branch is bounded by
$$
\iota^{-B_D/\epsilon^2}R_0^{50\epsilon^2}
\le R_0^{\epsilon/4}
$$
for sufficiently small $\varepsilon$. Together with
\eqref{eq 3.82}, this gives the stated $R_0^{\epsilon/2}$ terminal term.
This proves the proposition.
\end{proof}

\begin{remark}\label{remark}
    In the process of applying Proposition \ref{prop downscale}, each step reduces the scale by a factor of $K_{2}$, and the number of such choices is at most $O_{\epsilon}(1)$. Moreover, in the application of Proposition \ref{prop NB}, there are $O_{\epsilon}(1)$ choices for reducing the scale by $K_{1}$. Consequently, the total number of possible choices is at most $O_{\epsilon}(1)^{O_{\epsilon}(1)}=O_{\epsilon}(1)$.
\end{remark}

We now return to the original pair $(\Omega_1,\Omega_2)$ and apply the preceding multiscale analysis at the top frequency scale.

\begin{prop}\label{prop bigscale}
 Let $(\Omega_{1},\Omega_{2})$ be $\frac{1}{K_{2}}$-good pair. Fix some $R^{\frac{1}{2}}$-balls $Q$. Then one of the following is true:
 \begin{enumerate}
     \item there is $K_{1}R^{-\frac{1}{2}}\lesssim r\lesssim K_2^{-1}$ and a family of pairwise disjoint $(\frac{r}{K_{2}},r)$-rectangles $\omega_{i}\subset \Omega_{i}$, in generic position, such that
     \begin{equation}\label{eq 3.131}
         \int_{Q}|f_{\Omega_{1}}f_{\Omega_{2}}|^{2}\lesssim K_{1}^{O(1)}\sum_{\omega_{1}\subset \Omega_{1}}\|g_{\omega_{1}}\|_{L^{4}(Q)}^{2}\sum_{\omega_{2}\subset \Omega_{2}}\|g_{\omega_{2}}\|_{L^{4}(Q)}^{2}.
     \end{equation}
     \item We have
     \begin{equation}\label{eq 3.132}
         \int_{Q}|f_{\Omega_{1}}f_{\Omega_{2}}|^{2}\lesssim K_{2}^{O(1)}\int_{Q}\sum_{\theta_{1}\subset \Omega_{1}}|f_{\theta_{1}}|^{2}\sum_{\theta_2\subset \Omega_{2}}|f_{\theta_{2}}|^{2},
     \end{equation}
     where $\{\theta_{1}\},\{\theta_{2}\}$ are collection of squares with scale $R^{-\frac{1}{2}}$.
 \end{enumerate}    
\end{prop}

\begin{proof}
 For $\Omega_{1},\Omega_{2}$, we may take $r=\frac{1}{K_{2}}$. We verify that $d\sim 1\ge {K_{2}}r$ and $d{R^{\frac{1}{2}}}r\geq K_{2}$. Thus, the pair $(\Omega_{1},\Omega_{2})$ is broad or narrow. 
  
 If the pair $(\Omega_{1},\Omega_{2})$ is broad, we have the following inequality by H\"{o}lder inequality and Minkowski inequality
 \begin{equation}
 \begin{aligned}
      \int_{Q}|f_{\Omega_{1}}f_{\Omega_{2}}|^{2}&\lesssim K_{1}^{O(1)}\int_{Q} \sum_{\omega_{1}\subset \Omega_{1}}|g_{\omega_{1}}|^{2}\sum_{\omega_{2}\subset \Omega_{2}}|g_{\omega_{2}}|^{2}\\
      & \lesssim K_{1}^{O(1)}\Big(\int_{Q}(\sum_{\omega_{1}\subset \Omega_{1}}|g_{\omega_{1}}|^{2})^{2}\Big)^{\frac{1}{2}}\Big(\int_{Q}(\sum_{\omega_{2}\subset \Omega_{2}}|g_{\omega_{2}}|^{2})^{2}\Big)^{\frac{1}{2}}\\
      &\lesssim K_{1}^{O(1)}\sum_{\omega_{1}\subset \Omega_{1}}\|g_{\omega_{1}}\|_{L^{4}(Q)}^{2}\sum_{\omega_{2}\subset \Omega_{2}}\|g_{\omega_{2}}\|_{L^{4}(Q)}^{2},
 \end{aligned}
 \end{equation}
where $\omega_{i}$ are $(\frac{1}{K_{2}^{2}},\frac{1}{K_{2}})$-rectangles in generic position. Note that here we apply Lemma \ref{lem dichotomy} since $d\gtrsim K_{2}r$ and $dR^{\frac{1}{2}}r\geq K_{2}$.
Thus, we obtain the inequality \eqref{eq 3.131}.

 One can verify that the broad-narrow dichotomy will continue until the scale of squares $r \sim K_{2} \frac{1}{R^{\frac{1}{2}}}$. Let the pair $(\Omega_{1},\Omega_{2})$ be narrow. We assume the narrow case will continue at $m$ steps until occurs broad case, where $m\ge 1$. Then we obtain
\begin{equation}
    \int_{Q}|f_{\Omega_{1}}f_{\Omega_{2}}|^{2}\lesssim \sum_{\sigma_{1}\subset \Omega_{1}}\sum_{\sigma_{2}\subset \Omega_{2}}\int_{Q}|f_{\sigma_{1}}f_{\sigma_{2}}|^{2},
\end{equation}
where $\sigma_{i}$ are $\frac{1}{K_{2}K_{1}^{m}}$-squares. Let $r=\frac{1}{K_{2}K_{1}^{m}}$.
 If $r \sim \frac{K_{2}}{R^{\frac{1}{2}}}$, we apply the triangle inequality to obtain \eqref{eq 3.132}. If $r \gtrsim K_{2} \frac{1}{R^{\frac{1}{2}}}$, meaning that the broad case occurs before $r$ reaches the bottom scale, we repeat the previous computations to obtain \eqref{eq 3.131}.
\end{proof}

{
\begin{remark}\label{ref two-sided}
    In the bilinear refined decoupling setting, since the Fourier support of $f_{1}$ and $f_{2}$ is a good pair, then $(\Omega_{1},\Omega_{2})$ is a good pair. Thus,
    both $\omega_{1}$ and $\omega_{2}$ are in a generic position by Lemma \ref{lem:curvature-transfer} and stability analysis. 
\end{remark}
}

The next proposition upgrades the local estimate on a single $R^{\frac{1}{2}}$-ball to a global estimate over a collection of such balls, while keeping track of the wave packet multiplicity.

\begin{prop}\label{prop compute}
    Let $X$ be a family of $R^{\frac{1}{2}}$-balls $Q$.  
    Let $\Omega \subset [-1,1]^{2}$ be a square, and suppose $f = \sum_{T \in \mathbb{T}} f_{T}$ is a sum of wave packets with $\operatorname{supp}(\hat f) \subset \mathcal{N}_{\frac{1}{R}}(\mathcal{S}_{\Omega})$.  
    Assume there exists $M \ge 1$ such that each $Q \in X$ meets at most $M$ of the $R$-tubes in $\mathbb{T}$.  
    Let $K_1R^{-1/2}\lesssim r\lesssim K_2^{-1}$, and consider a collection of pairwise disjoint $\bigl(\frac{r}{K_{2}}, r\bigr)$-rectangles $\omega \subset \Omega$ in generic position.  
    Then
    \begin{equation}\label{eq long}
         \begin{aligned}
      \sum_{Q \subset X} \Bigl( \sum_{\omega} \|g_{\omega}\|_{L^4(Q)}^2 \Bigr)^2
      &\lesssim \Bigl( K_{1}^{O(1)} (\log(2R_0))^{O(1)} 
        \sup_{K_1R^{-1/2}\lesssim r'\lesssim K_3^{-1}} \bigl(\tfrac{r}{r'}\bigr)^{200 \frac{\log K_{1}}{\log K_{2}}}
        \overline C\bigl(R (r')^{2};\iota \bigr) \\
      &\qquad + R_0^{\epsilon} \Bigr) M \sum_{T \in \mathbb{T}} \|f_{T}\|_{4}^4 .
    \end{aligned}
    \end{equation}
\end{prop}

\begin{proof}
Applying Proposition \ref{prop downscale} to each $Q$, we obtain 
\begin{equation}\label{eq 1}
    \sum_{\omega}\|g_{\omega}\|_{L^{4}(Q)}^{2}\lesssim_{\epsilon}\iota^{-B_D/\epsilon^{2}}\sum_{\omega'}\Bigl(\frac{r}{r'}\Bigr)^{100\frac{\log K_{1}}{\log K_{2}}}\|g_{\omega'}\|_{L^{4}(Q)}^{2}+R_0^{\epsilon/2}\sum_{\theta\subset \Omega} \|f_{\theta}\|_{L^{4}(Q)}^{2},
\end{equation}
where $K_{1}R^{-\frac{1}{2}}\lesssim r'\lesssim K_{3}^{-1}$ and the $\omega'$ are pairwise disjoint $(\frac{r'}{K_{2}},r')$-rectangles in $\Omega$ in generic position. 

For the second term in the above inequality, since the $\{\theta\}$ are $R^{-\frac{1}{2}}$-squares, we have
\begin{equation}\label{eq 2}
    \sum_{Q\subset X}(\sum_{\theta\subset \Omega} \|f_{\theta}\|_{L^{4}(Q)}^{2})^{2}\lesssim M\sum_{T\in \mathbb{T}}\|f_{T}\|_{4}^{4}.
\end{equation}

We now deal with the first term. By pigeonholing, we may assume that each $Q$ has the same collection $\{\omega'\}$ via Remark \ref{remark}. By another pigeonholing, we may also assume that, for some fixed $N$, each $Q$ receives a $(\log R)^{-O(1)}$-fraction of the contribution to the integral from $\sim M/N$ tubes, from each of $\sim N$ rectangles $\omega'$. For such a pair, we write $Q\sim\omega'$. Then 
\begin{align*}
\sum_{Q}\;\Bigl(\sum_{\omega'}\|g_{\omega'}\|_{L^4(Q)}^2\Bigr)^2&\lesssim (\log(2R_0))^{O(1)} N\sum_{Q}\sum_{\omega'\sim Q}\|g_{\omega'}\|_{L^4(Q)}^4\\
&=(\log(2R_0))^{O(1)} N\sum_{\omega'}\|g_{\omega'}\|_{L^4(\cup_{Q\sim\omega'}Q)}^4.
\end{align*}
Recall that $g_{\omega'}=(\sum_{s\not\sim s'\subset \omega'}|f_{s}f_{s'}|)^{1/2}$, then we compute
$$\|g_{\omega'}\|_{L^{4}(\bigcup_{Q\sim \omega'}Q)}^{4}=\int_{\bigcup_{Q\sim \omega'}Q} \Bigl(\sum_{s\not\sim s'\subset \omega'}|f_{s}f_{s'}|\Bigr)^{2}\lesssim K_{1}^{O(1)}\int_{\bigcup_{Q\sim \omega'}Q}|f_{s}f_{s'}|^{2}.$$
By Proposition \ref{prop decouplingconstant},
\begin{equation}\label{eq 3}
    \|g_{\omega'}\|_{L^{4}(\bigcup_{Q\sim \omega'}Q)}^{4} \lesssim K_{1}^{O(1)}\iota^{-O(1)} \overline C\bigl(R(r')^{2};\iota\bigr) \frac{M}{N} \sum_{T\in \mathbb{T}_{\omega'}}\|f_{T}\|_{4}^{4}.
\end{equation}
Combining inequalities \eqref{eq 1}, \eqref{eq 2} and \eqref{eq 3}, we obtain \eqref{eq long}. Here the terminal factor becomes $R_0^\epsilon$ after squaring, and the total curvature loss is absorbed into $K_1^{O(1)}$, since
\[
\iota^{-2B_D/\epsilon^2-O_D(1)}
\le R_0^{(2B_D/\epsilon^2+O_D(1))\varepsilon^{9}}
\le R_0^{\epsilon^6}=K_1
\]
for sufficiently small $\varepsilon$, with $\epsilon=\varepsilon/10$.

\end{proof}

Combining the previous two propositions yields the desired recursive relation for the refined decoupling constant.

\begin{theorem}
With the parameters $R_0,\epsilon,K_1,K_2,K_3,\iota$ fixed as above,
for every $R_*<R\le R_0$ one has
    \begin{equation}\label{eq relation}
        C(R;\iota)\lesssim_{\epsilon,D}(K_{1}K_{2})^{O(1)}\Big( (\log(2R_0))^{O(1)}\sup_{K_1R^{-1/2}\lesssim r'\lesssim K_3^{-1}} \bigl(\tfrac{1}{r'}\bigr)^{200 \frac{\log K_{1}}{\log K_{2}}}
        \overline C\bigl(R (r')^{2};\iota \bigr) +R_0^{\epsilon} \Big).
    \end{equation}
\end{theorem}

\begin{proof}
Let $(\tau_{1},\tau_{2})$ be any good pair, and let $f_{1},f_{2},M_{1},M_{2},X$ be as in Definition \ref{def refineddecoupling}. 
We begin by decomposing each $\tau_i$ into $\frac{1}{K_2}$-squares $\Omega_i$. By the triangle inequality,
\begin{equation}
    \int_{X}|f_{1}f_{2}|^{2}\leq K_{2}^{O(1)}\sum_{\Omega_{1},\Omega_{2}}\int_{X}|f_{\Omega_{1}}f_{\Omega_{2}}|^{2}.
\end{equation}
We then fix $\Omega_{1},\Omega_{2}$ and apply Proposition \ref{prop bigscale} on each cube $Q\subset X$.

{\it Case 1:} Assume that \eqref{eq 3.131} holds for every $Q\subset X$. As in the previous argument, we may assume that the rectangle family $\{\omega_i\}$ does not depend on $Q$. 
Summing \eqref{eq 3.131} over all $Q\subset X$ and applying Cauchy--Schwarz, we get
\begin{equation}
    \int_{X}|f_{\Omega_{1}}f_{\Omega_{2}}|^{2}\lesssim K_{1}^{O(1)}
    \Big(\sum_{Q\subset X}\bigl(\sum_{\omega_{1}\subset \Omega_{1}}\|g_{\omega_{1}}\|_{L^{4}(Q)}^{2}\bigr)^{2} \Big)^{\frac{1}{2}}
    \Big(\sum_{Q\subset X}\bigl(\sum_{\omega_{2}\subset \Omega_{2}}\|g_{\omega_{2}}\|_{L^{4}(Q)}^{2}\bigr)^{2} \Big)^{\frac{1}{2}}.
\end{equation}
Next, we apply Proposition \ref{prop compute} to
\[
\Big(\sum_{Q\subset X}\bigl(\sum_{\omega_{1}\subset \Omega_{1}}\|g_{\omega_{1}}\|_{L^{4}(Q)}^{2}\bigr)^{2} \Big)^{\frac{1}{2}}
\quad\text{and}\quad
\Big(\sum_{Q\subset X}\bigl(\sum_{\omega_{2}\subset \Omega_{2}}\|g_{\omega_{2}}\|_{L^{4}(Q)}^{2}\bigr)^{2} \Big)^{\frac{1}{2}},
\]
which yields
\begin{equation}
    \begin{aligned}
        \int_{X}|f_{\Omega_{1}}f_{\Omega_{2}}|^{2}&\lesssim 
        \Big( (\log(2R_0))^{O(1)}\sup_{K_1R^{-1/2}\lesssim r'\lesssim K_3^{-1}} \bigl(\tfrac{1}{r'}\bigr)^{200 \frac{\log K_{1}}{\log K_{2}}}
        \overline C\bigl(R (r')^{2};\iota \bigr) +R_0^{\epsilon} \Big)\\
        &\qquad \times(K_{1}K_{2})^{O(1)} (M_{1}M_{2})^{\frac{1}{2}}
        \Bigl(\sum_{T\in \mathbb{T}_{\Omega_{1}}}\|f_{T}\|_{4}^{4}\Bigr)^{\frac{1}{2}}
        \Bigl(\sum_{T\in \mathbb{T}_{\Omega_{2}}}\|f_{T}\|_{4}^{4}\Bigr)^{\frac{1}{2}}.
    \end{aligned}
\end{equation}

{\it Case 2:} Assume instead that \eqref{eq 3.132} holds for every $Q\subset X$. In this case, inequality \eqref{eq 2} immediately gives
\[
C(R;\iota)\lesssim R_0^{\epsilon}.
\]
This completes the proof.
\end{proof}

\begin{proof}[The proof of Theorem \ref{thm refineddecoupling}]
We use the preceding local decomposition and rescaling estimates with the
fixed parameters specified above. First, the standard local wave-packet
count and H\"older's inequality give a crude bound
\[
C(t;\iota)\le A_Dt^B,\qquad t\ge1,
\]
for some fixed $B\ge1$, independently of $\iota$. Indeed, for the
locally relevant packets, if $N_j=\#\mathbb T_j$, then
\[
\int_X|f_1f_2|^2
\le (M_{1}M_{2})^{\frac{1}{2}}\frac{(N_1N_2)^{3/2}}{(M_{1}M_{2})^{\frac{1}{2}}}
\prod_{j=1}^2\left(\sum_{T\in\mathbb T_j}\|f_T\|_4^4\right)^{1/2}.
\]
The standard packet grid has $N_j\le A_Dt^{O(1)}$ locally; rapidly
decaying tails are summed separately. Since $M_j\ge1$, this proves
the crude bound after enlarging $B$.

Define the finite quantity
\[
\mathcal A:=\sup_{1\le t\le R_0}
\frac{C(t;\iota)}{R_0^\epsilon t^\epsilon}.
\]
For $t\le R_*=R_0^{L\epsilon^2}$, the crude bound and
$LB\epsilon^2\le\epsilon$ give
\[
\frac{C(t;\iota)}{R_0^\epsilon t^\epsilon}
\le A_DR_0^{LB\epsilon^2-\epsilon}\le A_D.
\]
For $t>R_*$, apply \eqref{eq relation} with $R=t$.
Let $P_0$ dominate its prefactor, so that, for fixed constants
$A_1,A_2$ depending only on $D$,
\[
P_0=A_{\epsilon,D}(K_1K_2)^{A_1}
       (\log(2R_0))^{A_2}
\le R_0^{c_{D}\epsilon^4},
\]
where $c_{D}$ is a constant independent of $\epsilon$.
 This uses $\log_{R_0}(K_1K_2)=\epsilon^6+\epsilon^4$;
the curvature losses have already been absorbed in
Proposition \ref{prop compute}.

For every admissible $r'$, the definition of $\mathcal A$ yields
\[
\overline C(t(r')^2;\iota)
\le\mathcal A R_0^\epsilon t^\epsilon(r')^{2\epsilon}.
\]
All such scales lie in $[1,R_0]$ by the lower and upper cutoffs
in \eqref{eq relation}. Hence
\[
\frac{C(t;\iota)}{R_0^\epsilon t^\epsilon}
\le P_0\sup_{r'\lesssim K_3^{-1}}
       (r')^{2\epsilon-200\epsilon^2}\mathcal A
       +P_0t^{-\epsilon}.
\]
Choose $\varepsilon$ small enough that
$2\epsilon-200\epsilon^2\ge\epsilon$.
For sufficiently large $R_0$, the implicit fixed constant in
$r'\lesssim K_3^{-1}$ can be absorbed into $K_3^{1/2}$, and therefore
\[
P_0\sup_{r'\lesssim K_3^{-1}}
(r')^{2\epsilon-200\epsilon^2}
\le R_0^{c_{D}\epsilon^4}K_3^{-\epsilon/2}
=R_0^{-\epsilon^3/4}\le\frac14.
\]
Also, since $t>R_*$,
\[
P_0t^{-\epsilon}
\le R_0^{c_{D}\epsilon^4-L\epsilon^3}\le\frac14.
\]
Taking the supremum over $1\le t\le R_0$, we obtain
\[
\mathcal A\le\max\{A_D,\frac{\mathcal A}{4}+\frac{1}{4}\},
\]
so $\mathcal A\le\max\{A_D,1\}$. In particular,
\begin{equation}\label{eq epsilon/5}
    \overline C(R_0;\iota)\lesssim_{\varepsilon,D}R_0^{2\epsilon}
=R_0^{\varepsilon/5}\le R_0^\varepsilon.
\end{equation}
Finally, monotonicity in $\iota$ gives the
same conclusion for every $\iota\ge R_0^{-\varepsilon^{9}}$.
This completes the proof.
\end{proof}

\section{Restriction estimates}

In this section, we will prove Theorem \ref{thm restriction2} by combining a broad estimate and an induction-on-scales argument for the narrow part.
We first set up the wave packet decomposition and the incidence input.
Then we establish the broad estimate.
Finally, we combine the broad and narrow contributions to conclude the proof.

 Let $\mathcal{S}$ be a surface of the form \eqref{eq reduction}, and let $E$ denote the extension operator associated with $\mathcal{S}$ as in \eqref{eq extensionoperator} (where we write $E$ for $E_\mathcal{S}$ when no confusion arises). 
 
\subsection{Wave packet decomposition and incidence geometry}
In this subsection, we introduce the wave packet decomposition and record several standard properties that will be used throughout the paper.

Given the $\varepsilon_{3}$ in Theorem \ref{thm restriction2}, we fix the small constant
\begin{equation}
\label{fiourf8ur8gut8gu8tu}
\varepsilon_{3}=\epsilon^{1000}.
\end{equation}
In frequency space, let $\Theta$ be a finitely overlapping cover of $[-1,1]^2$ by $R^{-1/2}$-balls $\theta$, and let $\{\phi_\theta\}_{\theta\in \Theta}$ be a smooth partition of unity such that
\[
\supp(\phi_\theta)\subset 2\theta
\qquad\text{and}\qquad
\sum_{\theta\in\Theta}\phi_\theta=1
\quad\text{on }[-1,1]^2.
\]
For $f:[-1,1]^2\to\mathbb{C}$, we continue to write
\[
f_\theta=f\phi_\theta.
\]

In physical space, let $\mathcal{V}$ be a finitely overlapping cover of $\mathbb{R}^2$ by $R^{1/2}$-balls, and let $\{\psi_v\}_{v\in\mathcal{V}}$ be a smooth partition of unity on $\mathbb{R}^2$ such that $\hat \psi_v$ is concentrated near $v$,
\[
\text{supp}(\hat\psi_v)\subset B_{R^{-1/2}}^2(0),
\qquad\text{and}\qquad
\sum_{v\in\mathcal{V}}\psi_v=1
\quad\text{on }\mathbb{R}^2.
\]

These partitions induce the wave packet decomposition of any function $f$ supported on $[-1,1]^2$:
\begin{equation}
\nonumber
f=\sum_{\theta\in \Theta}\sum_{v\in\mathcal{V}}(f\phi_\theta)\ast\hat\psi_v
=:\sum_{(\theta,v)\in\Theta\times\mathcal{V}}f_{\theta,v}.
\end{equation}
For $x\in\mathbb{R}^3$, write $x=(\bar x,x_3)$, and let $\Phi$ be the form of \eqref{eq reduction}.
For each $\theta\in\Theta$ and $v\in\mathcal{V}$, define
\[
T_{\theta,v}
=
\Bigl\{(\bar x,x_3)\in B_R:
|\bar x-c_v+x_3\nabla\Phi(c_\theta)|
\le R^{1/2+\varepsilon_3}\Bigr\},
\]
which is a tube of dimensions $R^{1/2+\varepsilon_3}\times R^{1/2+\varepsilon_3}\times R$, where $c_\theta$ and $c_v$ denote the centers of $\theta$ and $v$, respectively. We denote by $V(\theta)$ the vector
\[
V(\theta)=(1,\nabla\Phi(c_\theta)).
\]
Let
\[
\bar{\mathbb{T}} (\theta)
=
\{T_{\theta,v}:v\in\mathcal{V}\text{ and }T_{\theta,v}\cap B_R\not=\varnothing\}
\]
be the family of $R$-tubes with direction $V(\theta)$, and set
\[
\bar{\mathbb{T}}=\bigcup_\theta \bar{\mathbb{T}}(\theta).
\]
If $T=T_{\theta,v}$, we write
\begin{equation}
\label{jiojgiotgiutu8hu8y}
f_T=f_{\theta,v}, \;\theta=\theta_T.
\end{equation}

The next lemma summarizes the standard properties of this decomposition.

\begin{lemma}\label{lem property}
\label{wpt}
The wave packet decomposition satisfies the following properties.
\begin{enumerate}
\item $Ef=\sum_{T\in\bar{\mathbb{T}}}Ef_T$.
\item $ |Ef_T(x)|\lesssim R^{-1000}$ for all $x\in B_R\setminus T$.
\item $\supp f_T\subset 3\theta$ whenever $T$ has direction $V(\theta)$.
\item The family $\{V(\theta)\}_{\theta\in\Theta}$ is $\gtrsim R^{-1/2}$-separated.
\item For each $\theta\in\Theta$, the collection $\bar{\mathbb{T}}(\theta)$ is $R^{O(\varepsilon_3)}$-overlapping.
\item For every $T\in\bar{\mathbb{T}}$,
\[
\|Ef_T\|_{L^p(w_{B_R})}
\lesssim
R^{2(\frac{1}{p}-\frac{1}{2})}
\|Ef_T\|_{L^2(w_{B_R})},
\]
where $w_{B_R}$ is a weight satisfying $w_{B_R}\sim1$ on $B_R$ and decaying rapidly outside $B_R$.
\end{enumerate}
\end{lemma}

Next, we recall the $L^{2}$ estimate for the extension operators, and we will use these $L^{2}$-estimates in the subsequent interpolation. The second estimate is a refined $L^{2}$ orthogonality, and the definition involving shading will be given later.

\begin{lemma}\label{lem TS}
    We have
    $$\|Ef\|_{L^{4}(\mathbb{R}^{3})}\lesssim \|f\|_{L^{2}([-1,1]^{2})}.$$
\end{lemma}

\begin{lemma}\label{lem L2}
    Let $X$ be a union of $R^{\frac{1}{2}}$-balls, and let $f=\sum_{T\in \mathbb{T}}f_{T}$ be a sum of wave packets. Suppose for each $T\in \mathbb{T}$ there is a shading $Y(T)\subset T$ by $R^{\frac{1}{2}}$-balls in $X$ such that the number of $R^{\frac{1}{2}}$-balls intersecting $Y(T)$ is $\lesssim \lambda R^{\frac{1}{2}}$. Then
    $$\int_{X}|\sum_{T\in \mathbb{T}}Ef_{T}\mathbf{1}_{Y(T)}|^{2}\lesssim \lambda R\|f\|_{2}^{2}. $$
\end{lemma}

Now we turn our perspective to incidence geometry and give the notions of shading and denseness.
\begin{definition}
Let $\delta \in (0,1)$. Let $\mathbb{T}$ be a family of $\delta$-tubes in $\mathbb{R}^{3}$. A shading $Y$ is an assignment $T\rightarrow{Y(T)}$ such that
$Y(T)$ is a union of $\delta$-balls lying inside $T\cap B_{1}^{3}$ for each $T\in \mathbb{T}$. We say $Y$ is $\lambda$-dense, if $|Y(T)|\geq \lambda|T\cap B_{1}^{3}|$.
\end{definition}

Let us recall some other notions from incidence geometry, namely the two-ends condition and the $m$-parallel condition. The former describes the non-concentration property of the shading within a tube, while the latter controls the number of parallel tubes in the tube family.

\begin{definition}[Two-ends]
Let $\delta\in (0,1)$ and $(\mathbb{T},Y)$ be a $\delta$-separated tube family with associated shading. Let $0<\varepsilon_{2}<\varepsilon_{1}<1$. We say the shading $Y$
is $(\varepsilon_{1},\varepsilon_{2})$-two-ends if for all $T\in \mathbb{T}$ and all $\delta\times \delta^{\varepsilon_{1}}$-tubes $J\subset T$,
$$|Y(T)\cap J|\lesssim \delta^{\varepsilon_{2}}|Y(T)|.$$
\end{definition}

\begin{definition}
    We say a family of $\delta$-tubes $\mathbb{T}$ is $m$-parallel if there are $\lesssim m$ tubes pointing in the direction $V$ for every $\delta$-ball in the unit sphere $\mathbb{S}^{2}$.
\end{definition}

We now recall the incidence estimate due to Wang and Wu, together with later refinements by Gao, Wu, and the third author. The argument uses the two-ends Furstenberg inequality and the hairbrush structure. 

\begin{lemma}[\cite{WangWu2024,GWX}]\label{lem incidence}
    Let $(\mathbb{T},Y)$ be a set of tubes and shading in $\mathbb{R}^{3}$, and let $m\ge1$. Suppose $\mathbb{T}$ is the tube family with $m$-parallel and $Y$ is an $(\varepsilon_{1},\varepsilon_{2})$-two-ends, $\lambda$-dense shading. For all $x\in \bigcup_{T\in \mathbb{T}}Y(T)$, define $\mathbb{T}(x)=\{T\in \mathbb{T}:x\in Y(T)\}$.
    For any $\epsilon>0$, there exists a constant $C_{\epsilon}$ such that
    \begin{equation}
        \bigl|\bigcup_{\mathbb{T}} Y(T)\bigr|\ge C_{\epsilon}\delta^{\epsilon}\delta^{O(\varepsilon_{1})}m^{-\frac{3}{4}}\lambda^{\frac{7}{4}}\delta^{\frac{1}{2}}(\delta^{2}\#\mathbb{T})^{\frac{3}{4}}.
    \end{equation}
    In particular, by taking $\mu=\delta^{-O(\varepsilon_{1})}m^{\frac{3}{4}}\lambda^{-\frac{3}{4}}\delta^{-\frac{1}{2}}(\delta^{2}\#\mathbb{T})^{\frac{1}{4}}$, there exists a set
    $X_{\mu}\subset \bigcup_{T\in \mathbb{T}}Y(T)$ such that $\#\mathbb{T}(x)\lessapprox \mu$ for all $x\in X_{\mu}$, and
    $$|\bigcup_{T\in \mathbb{T}}Y(T)\setminus X_{\mu}|\leq \delta^{\varepsilon_{1}}|\bigcup_{T\in \mathbb{T}}Y(T)|.$$
\end{lemma}

\subsection{Broad and narrow reduction}

We work with the normalized polynomial phase $h$ of degree at most $D$
from the preceding geometric discussion. Fix an integer
$M=M(D)>D+5$.
Let $K\gg1$ be dyadic, and set
\begin{equation}\label{eq:bn-scales}
\ell=\frac1K,
\qquad
\rho=K^{-1/M}=\ell^{1/M}.
\end{equation}
 Choose
\begin{equation}\label{eq:bn-thresholds}
\begin{aligned}
\nu&=c_0\varepsilon_0\rho^{D+2},\\
\iota&=c_1\nu\rho^2
     =c_0c_1\varepsilon_0K^{-(D+4)/M},
\end{aligned}
\end{equation}
where $c_0,c_1>0$ are sufficiently small constants depending only on
$D$, as in the quantitative coarse-covering and two-sided extraction
lemmas. Choose another sufficiently small constant $c_2=c_2(D)>0$.
Since $M>D+5$ and $K\gg1$, we have
\begin{equation}\label{eq:bn-stability-scale}
\rho\le10^{-2},
\qquad
\ell\le c_2\iota\rho.
\end{equation}

Let $\mathcal C_K$ denote the collection of dyadic $1/K$-squares
partitioning $\Omega=[-1,1]^2$, up to boundaries. Write $c_\tau$ for
the center of $\tau$ and $3\tau$ for its concentric triple dilate.
The following lemma explains the choice of scales. For each $\tau$, the squares that interact with $\tau$ with small curvature at both endpoints can be covered by $O_D(1)$ coarse squares and approximately affine bad strips.

\begin{lemma}\label{lem choice}
\label{lem:nu-coarse-cover}
For the above scales, with $c_0$ sufficiently small and $K$
sufficiently large, there are fixed families $\mathcal P_\rho$
and $\mathcal S_\rho$ of coarse squares and approximately affine
$\rho$-bad strips, respectively, rounded to unions of fine squares,
such that
\[
\sum_{B\in\mathcal P_\rho\cup\mathcal S_\rho}\mathbf1_B
\le m_0(D).
\]
For every $\tau\in\mathcal C_K$, there is
$\mathcal R(\tau)\subset\mathcal P_\rho\cup\mathcal S_\rho$
with $\#\mathcal R(\tau)\le C_D$ whose union covers every
$\tau'\in\mathcal C_K$ with $|c_{\tau'}-c_\tau|<4\rho$, or with
\[
|c_{\tau'}-c_\tau|\ge4\rho,
\qquad
\max\{|\kappa_{c_\tau}(c_{\tau'})|,
       |\kappa_{c_{\tau'}}(c_\tau)|\}\le\nu.
\]
The counting constants depend only on $D$.
\end{lemma}

\begin{proof}
Fix $p=c_\tau$. Points within distance $4\rho$ of $p$ are
covered by $O(1)$ coarse squares, so we consider the remaining
exceptional points.

Let $e$ be either null direction of $D^2h(p)$, and choose a
perpendicular unit vector $n$ so that
\[
F(x,y):=\partial_e h(p+xe+yn)-\partial_e h(p)
\]
satisfies $F_y\sim1$. We work on a fixed coordinate box containing
the relevant parameter domain, and denote its horizontal interval
by $J$. Since $|F_x|\lesssim_D\varepsilon_0$, the equation $F=0$
defines a graph $y=\phi(x)$ with
\[
\phi(0)=0,
\qquad
\phi'=-\frac{F_x}{F_y},
\qquad
|\phi'|\lesssim_D\varepsilon_0.
\]

Suppose $q=p+xe+yn$ satisfies
\[
|\kappa_p(q)|,\ |\kappa_q(p)|\le\nu.
\]
For one of the two choices of $e$, the direction $v(p,q)$ lies
within $O(\nu)$ of $e$, after choosing its sign. Consequently,
\[
|F(x,y)|+|F_x(x,y)|\lesssim_D\nu.
\]
Since $F_y\sim1$, it follows that
\[
|y-\phi(x)|\lesssim_D\nu,
\qquad
x\in I_\nu:=
\{x\in J:|F_x(x,\phi(x))|\le C_\star\nu\},
\]
where $C_\star=C_\star(D)$ is sufficiently large.

The graph over $I_\nu$ is a semialgebraic set defined by a
bounded number of polynomial conditions of degree $O(D)$.
By the standard connected-component bound for semialgebraic sets
of bounded complexity, the graph over $I_\nu$, and hence $I_\nu$
itself, has $O_D(1)$ connected components.

If every component of $I_\nu$ has length less than $\rho$,
the corresponding exceptional points are covered by
$O_D(1)$ squares of sidelength $C_D\rho$.

Otherwise, choose an interval $I\subset I_\nu$ of length
$\rho$, and let $b=\phi(x_I)$, where $x_I$ is its midpoint.
Since $|\phi'|\lesssim_D\nu$ on $I$,
\[
\sup_{x\in I}|F(x,b)|\lesssim_D\nu\rho.
\]
The polynomial $F(\cdot,b)$ has degree at most $D-1$.
Choose $D$ equally spaced points $x_1,\ldots,x_D$ in $I$, where the distance of distinct $x_{i}$
is $\frac{\rho}{D}$.
We can present the polynomial as
$$
F(x,b)
=
\sum_{i=1}^D F(x_i,b)
\prod_{j\ne i}\frac{x-x_j}{x_i-x_j}.
$$
Thus, 
\[
\begin{aligned}
\|F(\cdot,b)\|_{C^1(J)}
&\lesssim_D
\rho^{-(D-1)}\|F(\cdot,b)\|_{L^\infty(I)}\\
&\lesssim_D\nu\rho^{-(D-2)}.
\end{aligned}
\]
Using $F(0,0)=0$, $F_y\sim1$, and the bounded derivatives of
$F$, we obtain
\[
|b|
+\|F(\cdot,0)\|_{C^1(J)}
+\|\phi\|_{L^\infty(J)}
\lesssim_D\nu\rho^{-(D-2)}
=c_0\varepsilon_0\rho^4.
\]
Thus the entire exceptional branch lies in an $O_D(\rho)$-strip
around $p+\mathbb R e$. Moreover,
\[
\partial_e^2h(p+xe)=F_x(x,0),
\]
so the trace of $h$ on this line has second derivative bounded
by $C_Dc_0\varepsilon_0\rho^4$. Applying this argument to both
null directions covers the exceptional points by $O_D(1)$
coarse squares and at most two approximately affine strips.

We now construct the covering objects and $\rho$-families.
For squares, use an enlarged $\rho$-grid. For horizontal strips
of width $L_D\rho$, discretize their central lines $y=kx+b$
using
\[
k\in c_D\varepsilon_0\rho\,\mathbb Z,
\qquad
b\in c_D\rho\,\mathbb Z,
\qquad
|k|\le\frac{1}{10},\quad |b|\le2,
\]
and retain only those satisfying
\begin{equation}\label{eq pseudobad}
    \left\|
\frac{d^2}{ds^2}h(s,ks+b)
\right\|_{L^\infty([-1,1])}
\le\varepsilon_0\rho.
\end{equation}
Rounding changes this derivative by at most
$C_Dc_D\varepsilon_0\rho$. Hence sufficiently small $c_0,c_D$
and sufficiently large $L_D$ preserve the preceding covering
and the strip condition. Treat vertical strips identically. 

To verify bounded overlap, fix $z$ and set
\[
Q_z(k)=h_{xx}(z)+2kh_{xy}(z)+k^2h_{yy}(z).
\]
Every retained horizontal strip containing $z$ satisfies
\[
|Q_z(k)|\lesssim_D\varepsilon_0\rho.
\]
Since $Q_z'(k)\ge1$, the possible slopes lie in an interval of
length $O_D(\varepsilon_0\rho)$, containing only $O_D(1)$ grid
slopes. For each slope, only $O_D(1)$ grid intercepts produce
a strip containing $z$. Together with the vertical strips
and coarse squares, this proves bounded overlap.

Finally, replace each coarse object $B$ by
\[
[B]_\ell
:=
\bigcup_{\substack{\tau'\in\mathcal C_K\\
                   \tau'\cap B\ne\varnothing}}\tau'.
\]
Since $\ell\ll\rho$, this rounding preserves the covering,
and the same counting argument gives bounded overlap.
Denote the resulting families of rounded squares and strips by
$\mathcal P_\rho$ and $\mathcal S_\rho$, respectively.
For simplicity, we continue to refer to their elements as
$\rho$-squares and $\rho$-strips, although they are, by construction,
unions of fine $\ell$-squares.
Moreover, each set $[B]_\ell$ is regarded as a single element of the
corresponding family.

The objects selected for $p=c_\tau$, together with the coarse
squares covering $B(p,4\rho)\cap\Omega$, form the required
family $\mathcal R(\tau)$. We naturally have $\#\mathcal{R}(\tau)\leq C_{D}$.
\end{proof}

For distinct $\tau_1,\tau_2\in\mathcal C_K$, we say that
$(\tau_1,\tau_2)$ is \emph{quantitatively $1/K$-transverse} if
\begin{equation}\label{eq:quantitative-transversality}
\begin{aligned}
\operatorname{dist}(3\tau_1,3\tau_2)&\gtrsim \rho,\\
\inf_{\substack{\zeta_1\in3\tau_1\\\zeta_2\in3\tau_2}}
\min_{j=1,2}
\left|
\partial_{v(\zeta_1,\zeta_2)}^2h(\zeta_j)
\right|&\gtrsim \iota.
\end{aligned}
\end{equation}

\begin{definition}[$A$-Broad]
\label{def:A-broad}
Let $A\ge1$. We say that a collection
$\mathcal M=\{\tau\}\subset\mathcal C_K$ is $A$-broad if
\begin{itemize}
\item $\#\mathcal M\ge A$;
\item for any distinct $\tau_1,\tau_2\in\mathcal M$, the pair
$(\tau_1,\tau_2)$ is quantitatively $1/K$-transverse in the sense of
\eqref{eq:quantitative-transversality}.
\end{itemize}
\end{definition}

\begin{definition}[Broad function]
\label{def:broad-function}
Let $A\ge1$, and consider a collection
$\{F^\tau\}_{\tau\in\mathcal C_K}$ of functions
$F^\tau:\mathbb R^3\to\mathbb C$.
For $x\in\mathbb R^3$, define
\[
\operatorname{Br}_A\{F^\tau\}(x)
=
\max_{\substack{\mathcal M\subset\mathcal C_K\\
                 \mathcal M\text{ is }A\text{-broad}}}
\min_{\tau\in\mathcal M}|F^\tau(x)|.
\]
\end{definition}

\begin{remark}
The function $\operatorname{Br}_A\{F^\tau\}(x)$ is non-increasing as
$A$ increases. 
\end{remark}

We record two elementary properties of the broad function.

\begin{lemma}\label{lem broad1}
Suppose that $A_j\ge1$ for $1\le j\le N$,
\[
A=A_1+\cdots+A_N,
\qquad
F^\tau=F_1^\tau+\cdots+F_N^\tau.
\]
Then, for every $x\in\mathbb R^3$,
\begin{equation}\label{eq triangleinequality}
\operatorname{Br}_A\{F^\tau\}(x)
\le
\sum_{j=1}^N
\operatorname{Br}_{A_j}\{F_j^\tau\}(x).
\end{equation}
\end{lemma}

\begin{lemma}\label{lem broad2}
If $A\ge2$, then
\begin{equation}\label{eq:broad-bilinear}
\operatorname{Br}_A\{F^\tau\}(x)
\le
\max_{\substack{\tau_1,\tau_2\in\mathcal C_K\\
(\tau_1,\tau_2)\text{ satisfies }
\eqref{eq:quantitative-transversality}}}
|F^{\tau_1}(x)F^{\tau_2}(x)|^{1/2},
\end{equation}
where an empty maximum is interpreted as zero.
\end{lemma}

Let $\Omega=[-1,1]^2$. Set
\[
 g=\nabla h,\qquad H=D^2h,\qquad
 J=\begin{pmatrix}0&1\\-1&0\end{pmatrix},
\]
and, for $p\ne q$, define
\[
 \kappa_p(q)=\frac{(Jw)^TH(p)(Jw)}{|w|^2},
 \qquad w=g(q)-g(p).
\]

We now present a lemma that extracts a lower bound for bilateral curvature. 
Roughly speaking, the lemma states that given a family of squares for which, 
between any two members, the curvature along their intersection directions does not vanish 
on at least one side, one can extract a subfamily such that for any two squares 
in the subfamily, the curvature on both sides along the intersection direction 
admits a uniform lower bound.

\begin{lemma}[Two-sided curvature extraction]\label{lem two}
Let $0<\rho,\nu\le10^{-2}$ and let $P\subset\Omega$ consist of $N\ge1$ points satisfying
\[
 |p-q|\ge\rho,\qquad
 \max\{|\kappa_p(q)|,|\kappa_q(p)|\}>\nu
 \qquad(p,q\in P,\ p\ne q).
\]
For a sufficiently small absolute constant $c_1>0$, set
$\iota=c_1\nu\rho^2$. Then there exists $P'\subset P$ such that
\[
 \#P'\ge\frac{\sqrt N}{7},\qquad
 \min\{|\kappa_p(q)|,|\kappa_q(p)|\}>4\iota
 \quad(p,q\in P',\ p\ne q).
\]
\end{lemma}

\begin{proof}
 We can check that
\begin{equation}\label{eq:gradient-comparison}
    |g(p)-g(q)|\sim |p-q|,
    \qquad
    \diam g(\Omega)\lesssim 1.
\end{equation}

For each $p\in P$, consider the quadratic form
$$
Q_p(w)=(Jw)^T H(p)(Jw).
$$
Since the two eigenvalues of $H(p)$ have opposite signs and
their absolute values are uniformly comparable to $1$, the
quadratic form $Q_p$ has two uniformly separated null directions.
Let $L_p^+$ and $L_p^-$ be the corresponding lines through $g(p)$.

For $q\in P\setminus\{p\}$, set $w=g(q)-g(p)$.
Choose orthonormal coordinates in the gradient plane centered
at $g(p)$ such that, writing $w=(s,t)$, we have
$$
Q_p(w)
=
\lambda s^2-\mu t^2
=
(\sqrt{\lambda}s-\sqrt{\mu}t)
(\sqrt{\lambda}s+\sqrt{\mu}t),
\qquad \lambda,\mu\sim 1.
$$
In these coordinates, $g(q)$ is represented by $(s,t)$, and
the two linear factors define the two null lines.
Set
$$
d_\pm:=\dist(g(q),L_p^\pm).
$$
By the point-to-line distance formula, after relabeling the
two lines if necessary, we obtain
$$
d_+
=
\frac{|\sqrt{\lambda}s-\sqrt{\mu}t|}{\sqrt{\lambda+\mu}},
\qquad
d_-
=
\frac{|\sqrt{\lambda}s+\sqrt{\mu}t|}{\sqrt{\lambda+\mu}}.
$$
Since $\lambda,\mu\sim 1$, it follows that
$$
|Q_p(w)|=(\lambda+\mu)d_+d_-\sim d_+d_-,
\qquad
\max\{d_+,d_-\}\sim |w|.
$$
Consequently,
$$
|\kappa_p(q)|
=
\frac{|Q_p(w)|}{|w|^2}
\sim
\frac{d_+d_-}{|w|^2}
\sim
\frac{\min\{d_+,d_-\}}{|w|}.
$$
In particular,
\begin{equation}\label{eq:edge1}
    \dist(g(q),L_p^+\cup L_p^-)
    \lesssim
    |\kappa_p(q)|\,|g(q)-g(p)|.
\end{equation}
Conversely, since $\min\{d_+,d_-\}\le d_\pm$, we also obtain,
for either of the two null lines,
\begin{equation}\label{eq:edge2}
    |\kappa_p(q)|
    \lesssim
    \frac{\dist(g(q),L_p^\pm)}
    {|g(q)-g(p)|}.
\end{equation}

We now put a directed edge
$$
p\longrightarrow q
$$
whenever $p\ne q$ and
$$
|\kappa_p(q)|\le 4\iota.
$$
By \eqref{eq:edge1} and the bounded diameter in
\eqref{eq:gradient-comparison}, an edge $p\to q$ implies
\begin{equation}\label{eq:edge-near-null}
    \dist(g(q),L_p^+\cup L_p^-)\lesssim\iota.
\end{equation}

\emph{\textbf{Claim.} Any two distinct vertices have at most four
common out-neighbors.}

Suppose otherwise that $p$ and $q$ have five common out-neighbors.
For each common out-neighbor, choose one of $L_p^+,L_p^-$ and one
of $L_q^+,L_q^-$ so that its gradient image lies within $C\iota$
of both chosen lines. Since there are only four possible pairs of
lines, the pigeonhole principle gives two distinct common
out-neighbors $r,s$ assigned to the same pair
$$
L_p^\alpha,\qquad L_q^\beta.
$$
Thus both $g(r)$ and $g(s)$ lie within $C\iota$ of both lines.

On the other hand, by the separation assumption and
\eqref{eq:gradient-comparison},
\begin{equation}\label{eq:rs-separation}
    |g(r)-g(s)|\gtrsim |r-s|\ge\rho.
\end{equation}
It follows from \eqref{eq:edge-near-null} and
\eqref{eq:rs-separation} that the angle between the two unoriented
lines satisfies
\begin{equation}\label{eq:null-line-angle}
    \angle(L_p^\alpha,L_q^\beta)
    \lesssim\frac{\iota}{\rho}.
\end{equation}

Since the two lines are within $O(\iota)$ of one another near $g(r)$,
their angle is controlled by \eqref{eq:null-line-angle}, and all
gradient points lie in a fixed bounded region by
\eqref{eq:gradient-comparison}, we obtain
\begin{equation}\label{eq:endpoint-null-distance}
    \dist(g(q),L_p^\alpha)
    +
    \dist(g(p),L_q^\beta)
    \lesssim
    \frac{\iota}{\rho}.
\end{equation}
Applying \eqref{eq:edge2}, \eqref{eq:gradient-comparison}, and
\eqref{eq:endpoint-null-distance}, we get
\begin{equation}\label{eq:curvature-contradiction}
\begin{aligned}
    |\kappa_p(q)|+|\kappa_q(p)|
    &\lesssim
    \frac{1}{|g(q)-g(p)|}
    \left(
    \dist(g(q),L_p^\alpha)
    +
    \dist(g(p),L_q^\beta)
    \right) \\
    &\lesssim
    \frac{\iota}{\rho^2}
    =
    c_1\nu.
\end{aligned}
\end{equation}
If $c_1$ is chosen sufficiently small, \eqref{eq:curvature-contradiction}
contradicts
$$
\max\{|\kappa_p(q)|,|\kappa_q(p)|\}>\nu.
$$
This proves the claim.

Let $e$ denote the number of directed edges and let $d^-(r)$ denote
the indegree of $r$. Since every pair of vertices has at most four
common out-neighbors, double counting gives
\begin{equation}\label{eq:common-neighbor-count}
    \sum_{r\in P}\binom{d^-(r)}{2}
    \le
    4\binom{N}{2},
\end{equation}
where $P$ denotes the family of vertexes and $N$ denotes the $\#{P}$.
Since
$$
\sum_{r\in P}d^-(r)=e,
$$
we obtain from \eqref{eq:common-neighbor-count}
\begin{equation}\label{eq:degree-square-bound}
\begin{aligned}
    \sum_{r\in P}d^-(r)^2
    &=
    \sum_{r\in P}d^-(r)
    +
    2\sum_{r\in P}\binom{d^-(r)}2 \\
    &\le
    e+4N(N-1).
\end{aligned}
\end{equation}
By Cauchy--Schwarz,
\begin{equation}\label{eq:directed-edge-bound}
    \frac{e^2}{N}
    \le
    \sum_{r\in P}d^-(r)^2
    \le
    e+4N(N-1).
\end{equation}
In particular,
\begin{equation}\label{eq:e-bound}
    e\le 3N^{3/2}.
\end{equation}

We now form an undirected graph by joining $p$ and $q$ whenever
either $p\to q$ or $q\to p$ occurs. If $m$ denotes its number of
edges, then by \eqref{eq:e-bound},
\begin{equation}\label{eq:m-bound}
    m\le e\le 3N^{3/2}.
\end{equation}

Let $d(p)$ denote the degree of $p$ in this undirected graph.
Choose a uniformly random ordering of all vertices and retain every
vertex which appears before all of its neighbors. The retained
vertices form an independent set. Moreover, the probability that a
given vertex $p$ is retained is
$$
\frac{1}{d(p)+1}.
$$
Hence there exists an independent set $P'$ satisfying
\begin{equation}\label{eq:independent-first}
    \#P'
    \ge
    \sum_{p\in P}\frac{1}{d(p)+1}.
\end{equation}

Since
$$
\sum_{p\in P}d(p)=2m,
$$
Cauchy--Schwarz gives
\begin{equation}\label{eq:caro-wei}
\begin{aligned}
    N^2
    &=
    \left(
    \sum_{p\in P}
    \frac{1}{\sqrt{d(p)+1}}
    \sqrt{d(p)+1}
    \right)^2 \\
    &\le
    \left(
    \sum_{p\in P}\frac{1}{d(p)+1}
    \right)
    \left(
    \sum_{p\in P}(d(p)+1)
    \right) \\
    &=
    (N+2m)
    \sum_{p\in P}\frac{1}{d(p)+1}.
\end{aligned}
\end{equation}
Combining \eqref{eq:independent-first}, \eqref{eq:m-bound}, and
\eqref{eq:caro-wei}, we obtain
\begin{equation}\label{eq:independent-final}
\begin{aligned}
    \#P'
    &\ge
    \frac{N^2}{N+2m} \\
    &\ge
    \frac{N^2}{N+6N^{3/2}}
    =
    \frac{\sqrt N}{6+N^{-1/2}}
    \ge
    \frac{\sqrt N}{7}.
\end{aligned}
\end{equation}

Finally, since $P'$ is an independent set, for any distinct
$p,q\in P'$ neither $p\to q$ nor $q\to p$ occurs. Hence
$$
|\kappa_p(q)|>4\iota,
\qquad
|\kappa_q(p)|>4\iota,
$$
and therefore
$$
\min\{|\kappa_p(q)|,|\kappa_q(p)|\}>4\iota.
$$
This completes the proof.
\end{proof}

We now introduce the narrow frequency sets. Let $\mathcal{P}_\rho$
be the fixed coarse family that consists of $\rho$-squares and $\mathcal{S}_{\rho}$ be the family of $\rho$-bad strips. The cover satisfies
\begin{equation}\label{eq:coarse-overlap-bn}
\sum_{B\in\mathcal P_\rho\bigcup\mathcal{S}_{\rho}}\mathbf 1_B\le m_0,
\qquad m_0=m_0(D).
\end{equation}

Morally speaking, the narrow part consists of those functions whose frequency support lies in certain $\rho$-squares and $\rho$-bad strips. However, if the narrow part in the broad-narrow reduction following is directly treated as a function of frequency support $\rho$-squares and $\rho$-bad strips, cancellation may occur. To be rigorous, we define the intersection family
\begin{equation}\label{eq:narrow-family-bn}
\mathcal N_\rho
=
\left\{
B_1\cap\cdots\cap B_j:
\begin{array}{l}
1\le j\le m_0,\quad
B_1,\ldots,B_j\in\mathcal P_\rho\bigcup \mathcal{S}_{\rho}\text{ distinct},\\[2pt]
B_1\cap\cdots\cap B_j\neq\varnothing
\end{array}
\right\},
\end{equation}
where identical intersections are counted only once.
$\mathcal{N}_{\rho}$ represents the intersection of a family of $\rho$-strips and $\rho$-squares. Since $j$ can take an index, this intersection also includes individual $\rho$-squares and $\rho$-strips.

\begin{prop}[Broad-narrow reduction]
\label{prop:broad-narrow-reduction}
For every integer $A\ge2$,
\begin{equation}\label{eq:broad-narrow-reduction}
|Ef(x)|
\lesssim_D
A^{2m_0}\max_{N\in\mathcal N_\rho}|Ef_N(x)|
+
K^2\operatorname{Br}_A Ef(x),
\end{equation}
where $f_N=f\mathbf1_N$ and an empty maximum is interpreted as zero.
In particular, taking $A=\lceil K^{\varepsilon'}\rceil$, we obtain,
for every $\varepsilon'>0$,
\[
|Ef(x)|
\lesssim_D
K^{2m_0\varepsilon'}
\max_{N\in\mathcal N_\rho}|Ef_N(x)|
+
K^2\operatorname{Br}_{\lceil K^{\varepsilon'}\rceil}Ef(x).
\]
\end{prop}

\begin{proof}
Fix $x\in\mathbb R^3$.
If $Ef(x)=0$, the conclusion is immediate, so assume $Ef(x)\neq0$.
Since
\[
n_K:=\#\mathcal C_K=4\ell^{-2}=4K^2,
\]
let $\mathcal M\subset\mathcal C_K$ be the collection of squares
satisfying
\begin{equation}\label{eq:significant-squares-bn}
|Ef_\tau(x)|\ge\frac{|Ef(x)|}{4n_K}.
\end{equation}
The remaining squares satisfy
\begin{equation}\label{eq:insignificant-squares-bn}
\sum_{\tau\in\mathcal C_K\setminus\mathcal M}|Ef_\tau(x)|
\le\frac{|Ef(x)|}{4}.
\end{equation}

For each $\tau\in\mathcal C_K$, let
$\mathcal R(\tau)\subset\mathcal P_\rho\cup\mathcal S_\rho$
be the subcollection supplied by Lemma \ref{lem choice}.
Its union covers every $\tau'\in\mathcal C_K$ satisfying either
\[
|c_{\tau'}-c_\tau|<4\rho,
\]
or
\[
\begin{gathered}
|c_{\tau'}-c_\tau|\ge4\rho,\\
\max\left\{
\left|\partial_{v(c_\tau,c_{\tau'})}^2h(c_\tau)\right|,
\left|\partial_{v(c_\tau,c_{\tau'})}^2h(c_{\tau'})\right|
\right\}\le\nu.
\end{gathered}
\]
Moreover, $\#\mathcal R(\tau)\le C_D$, where $C_D$ depends only on $D$.

We select squares from $\mathcal M$ greedily.
Starting with $U_0=\varnothing$, after choosing
$\tau_1,\ldots,\tau_m$, put
\[
U_m=\bigcup_{i=1}^m\ \bigcup_{B\in\mathcal R(\tau_i)}B.
\]
If $m<49A^2$ and there exists some $\tau\in\mathcal M$
that is not contained in $U_m$, choose such a square as
$\tau_{m+1}$; otherwise stop.
The selection implies that any distinct selected squares satisfy
\begin{equation}\label{eq:weak-selected-pairs-bn}
\begin{aligned}
|c_{\tau_i}-c_{\tau_j}|&\ge4\rho,\\
\max\left\{
\left|\partial_{v(c_{\tau_i},c_{\tau_j})}^2h(c_{\tau_i})\right|,
\left|\partial_{v(c_{\tau_i},c_{\tau_j})}^2h(c_{\tau_j})\right|
\right\}&>\nu.
\end{aligned}
\end{equation}
We distinguish two cases.

\medskip
\noindent
\emph{Case 1: the selection reaches $49A^2$ squares.}

By Lemma \ref{lem two}, there is a subcollection
$\mathcal M(x)$ of cardinality at least
$
\frac{\sqrt{49A^2}}{7}=A
$
such that every distinct $\tau_i,\tau_j\in\mathcal M(x)$ satisfies
\[
\min\left\{
\left|
\partial_{v(c_{\tau_i},c_{\tau_j})}^2h(c_{\tau_i})
\right|,
\left|
\partial_{v(c_{\tau_i},c_{\tau_j})}^2h(c_{\tau_j})
\right|
\right\}>4\iota.
\]
Since the fine sidelength satisfies $\ell\le c_2\iota\rho$, the
center-to-square stability estimate upgrades these center bounds to
\eqref{eq:quantitative-transversality} on the triple dilates,
provided $c_2$ is sufficiently small.
Thus these squares are pairwise quantitatively $1/K$-transverse.
After discarding squares if necessary, assume
$\#\mathcal M(x)=A$.
Every square in this collection is significant, and hence
\[
\operatorname{Br}_A Ef(x)
\ge\min_{\tau\in\mathcal M(x)}|Ef_\tau(x)|
\ge\frac{|Ef(x)|}{4n_K}.
\]
Consequently,
\begin{equation}\label{eq:broad-alternative-bn}
|Ef(x)|\le16K^2\operatorname{Br}_A Ef(x).
\end{equation}

\medskip
\noindent
\emph{Case 2: the selection stops before $49A^2$ squares are chosen.}

Every square in $\mathcal M$ is then contained in a union
\[
U=B_1\cup\cdots\cup B_L,
\qquad
L\lesssim_D A^2,
\]
where the $B_i\in\mathcal P_\rho\cup\mathcal S_\rho$ are distinct.
Recall that each $B_i$ is a union of fine squares in $\mathcal C_K$.
Thus, up to boundaries, $\Omega\setminus U$ is a union of squares
in $\mathcal C_K\setminus\mathcal M$.
By \eqref{eq:insignificant-squares-bn},
\begin{equation}\label{eq:covering-union-contribution-bn}
|Ef_U(x)|
\ge |Ef(x)|-|Ef_{\Omega\setminus U}(x)|
\ge\frac{3|Ef(x)|}{4}.
\end{equation}

By the bounded-overlap property \eqref{eq:coarse-overlap-bn},
no point belongs to more than $m_0$ of the sets $B_i$.
Hence inclusion--exclusion gives
\[
Ef_U(x)
=
\sum_{j=1}^{m_0}(-1)^{j+1}
\sum_{1\le i_1<\cdots<i_j\le L}
Ef_{B_{i_1}\cap\cdots\cap B_{i_j}}(x),
\]
where empty intersections contribute zero.

Every nonempty intersection belongs to $\mathcal N_\rho$, and the
number of terms is bounded by
\[
\sum_{j=1}^{m_0}\binom Lj
\lesssim_D L^{m_0}
\lesssim_D A^{2m_0}.
\]
Using the triangle inequality, we obtain
\[
|Ef_U(x)|
\le
\sum_{j=1}^{m_0}
\sum_{1\le i_1<\cdots<i_j\le L}
\left|Ef_{B_{i_1}\cap\cdots\cap B_{i_j}}(x)\right|
\lesssim_D
A^{2m_0}\max_{N\in\mathcal N_\rho}|Ef_N(x)|.
\]
Together with \eqref{eq:covering-union-contribution-bn}, this gives
\begin{equation}\label{eq:narrow-alternative-bn}
|Ef(x)|
\lesssim_D
A^{2m_0}\max_{N\in\mathcal N_\rho}|Ef_N(x)|.
\end{equation}
The two alternatives prove \eqref{eq:broad-narrow-reduction}.
The final assertion follows from
$\lceil K^{\varepsilon'}\rceil\le2K^{\varepsilon'}$.
\end{proof}

We call a strip of length $1$ and width $\rho$ a \emph{pseudo-bad strip} if the second derivative of $h$ in its long direction is bounded by $\varepsilon_0\rho$. Since an intersection is contained in each of its constituent sets, every element of $\mathcal N_\rho$, although possibly polygonal, is contained in either a $\rho$-square or a $\rho$-pseudo-bad strip. Therefore, in the subsequent induction-on-scales argument, it suffices to consider functions $f$ supported in such a square or strip.

\subsection{The estimate for the broad part}
\begin{prop}\label{prop broadestimate}
    Assume $\epsilon\ll10^{-5}$ is small enough. For $R\gg 1$, let $K=R^{\epsilon^{10}}$.
    Then there exists $C_{\epsilon}>0$ such that for all $R\ge1$,
    \begin{equation}\label{eq broad}
        \int_{B_{R}}|Br_{A}Ef|^{p}\leq C_{\epsilon}K^{O(1)}R^{2\epsilon}\|f\|_{L^{2}}^{2p-{4}}\sup_{\theta:R^{-1/2}\text{-square}}\|f_{\theta}\|_{L^{2}_{avg}(\theta)}^{4-p}.
    \end{equation}
    for $p=\frac{22}{7}$ and $A\ge R^{\epsilon^{20}}$. Here, $\|f_{\theta}\|_{L^{2}_{avg}(\theta)}$ is defined as
    $$\|f_{\theta}\|_{L^{2}_{avg}(\theta)}^{2}\coloneqq |\theta|^{-1}\|f_{\theta}\|_{2}^{2}.$$
\end{prop}

\begin{proof}
 Fix $\epsilon>0$. We proceed by induction on $r$ to prove inequality \eqref{eq broad} for all $r \in [R^{\epsilon^{2}}, R]$ and $A \ge r^{\epsilon^{20}}$.
It suffices to prove the following inequality
\begin{equation}\label{eq broadinduction}
    \int_{B_{r}}|Br_{A}Ef|^{p}\leq C_{\epsilon}K^{O(1)}R^{\epsilon}r^{\epsilon}\|f\|_{L^{2}}^{2p-4}\sup_{\theta:r^{-1/2}\text{-square}}\|f_{\theta}\|_{L^{2}_{avg}(\theta)}^{4-{p}}.
\end{equation}

When $r = R^{\epsilon^{2}}$, this is the base case, and it follows from elementary inequalities.
Let $n$ denote the number of iteration steps, satisfying $n \sim \frac{\ln\epsilon}{\ln(1-\epsilon^{2})}$.

Assume that \eqref{eq broadinduction} holds for $r = R^{\frac{\epsilon^{2}}{(1-\epsilon^{2})^{k-1}}}$ with $k \ge 1$.
Now fix $r = R^{\frac{\epsilon^{2}}{(1-\epsilon^{2})^{k}}}$ and fix $B_{r}$.
Partition the $r$-tubes $\bar{\mathbb{T}}=\mathbb{T}\cup\mathbb{T}_{s}$, where $\mathbb{T}_{s}=\{T\in \bar{\mathbb{T}}:\|f_{T}\|_{2}\leq r^{-100}\|f\|_{2}\}$.
For $f=\sum_{T\in\mathbb{T}_{s}}f_{T}$, we naturally get the inequality \eqref{eq broad}. Thus, we only need to consider $\int_{B_{r}}|Br_{A}Ef'|^{p}$, where $f'=\sum_{T\in \mathbb{T}}f_{T}$. Next, we reduce to the case in which the wave-packet norms and multiplicities are comparable. We partition $\mathbb{T}=\cup_{\gamma,m}\mathbb{T}_{\gamma,m}$, where $\gamma,m \in [r^{-100},r^{10}]$ are dyadic numbers such that
the following holds
\begin{itemize}
    \item For all $T\in \mathbb{T}_{\gamma,m}$, $\|f_{T}\|_{2}\sim \gamma\|f\|_{2}$.
    \item For all $\theta$, either $\mathbb{T}_{\gamma,m}(\theta)=\varnothing$, or $\#\mathbb{T}_{\gamma,m}(\theta)\sim m$, where $\mathbb{T}_{\gamma,m}(\theta)$ denotes the subfamily of tubes with direction $\theta$ in $\mathbb{T}_{\gamma,m}$.
\end{itemize}
By dyadic pigeonholing and Lemma \ref{lem broad1}, there exists a $(\gamma,m)$ and $A_{g}\gtrapprox A$ such that
\begin{equation}
    \int_{B_{r}}|Br_{A}Ef'|^{p}\lessapprox \int_{B_{r}}|Br_{A_{g}}Eg|^{p},
\end{equation}
where $g=\sum_{T\in \mathbb{T}_{\gamma,m}}f_{T}$.
Continuing with the dyadic pigeonhole principle, we obtain a family $X$ of $r^{1/2}$-balls $Q$ such that
\begin{itemize}
    \item $\int_{Q}|Br_{A_{g}}Eg|^{p}$ is the same up to a constant multiple for all $Q\subset X$. 
    \item \begin{equation}
        \int_{B_{r}}|Br_{A_{g}}Eg|^{p}\lessapprox \int_{X}|Br_{A_{g}}Eg|^{p}.
    \end{equation}
\end{itemize}

\textbf{Step 1: Two-ends reduction}

We partition each $r$-tube $T \in \mathbb{T}_{\gamma,m}$ into tube segments $J$ of length $r^{1-\epsilon^{2}}$. 
Let $\mathcal{J}(T)$ denote the set of those segments that intersect $X$.
We then partition $\mathcal{J}(T) = \bigcup_{\lambda} \mathcal{J}_{\lambda}(T)$, where $\lambda \in \Lambda$ and $\Lambda$ denotes the family of dyadic numbers in $[r^{-1/2}, r^{-\epsilon^{2}}]$, such that for any $J \in \mathcal{J}_{\lambda}(T)$,
\[
|J \cap X| \sim \lambda |T|.
\]
Consequently,
\[
Eg = \sum_{\lambda} \sum_{T \in \mathbb{T}_{\gamma,m}} Ef_{T} \sum_{J \in \mathcal{J}_{\lambda}(T)} \mathbf{1}_{J}.
\]
We take partition $\mathbb{T}_{\gamma,m}=\cup\mathbb{T}_{\beta}$, where $\beta\in [1,r^{\epsilon^{2}}]$ is a dyadic number such that $\#\mathcal{J}_{\lambda}(T)\sim \beta$.
Let 
\[
F_{\lambda,\beta}^{\tau} = \sum_{T \in \mathbb{T}_{\beta}, \, \theta_{T} \in \tau} Ef_{T} \sum_{J \in \mathcal{J}_{\lambda}(T)} \mathbf{1}_{J}, \qquad \tau \in \mathcal{C}_{K}.
\]
By Lemma \ref{lem broad1} and repeated applications of the dyadic pigeonhole principle, we obtain a two-ends reduction as follows.
There exist uniformly dyadic numbers $\lambda, \beta$ and $A_{2} \gtrapprox A_{g}$, as well as a set of $r^{1/2}$-balls $X_{2}$ such that
\begin{itemize}
    \item $|X_{2}| \gtrapprox |X|$;
    \item $\displaystyle \int_{X} |Br_{A_{g}} Eg|^{p} \lessapprox \int_{X_{2}} \bigl| Br_{A_{2}} \{ F_{\lambda,\beta}^{\tau} \} \bigr|^{p}.$
    \item $\displaystyle \int_{Q} \bigl| Br_{A_{2}} \{ F_{\lambda,\beta}^{\tau} \} \bigr|^{p}$ is the same up to a constant multiple for any $r^{1/2}$-ball $Q \subset X_{2}$.
\end{itemize}
Let $\{B_{v}\}$ be a finitely overlapping family of $r^{1-\epsilon^{2}}$-balls that cover $B_{r}$ and we only need to analyze $\int_{X_{2}}\bigl| Br_{A_{2}} \{ F_{\lambda,\beta}^{\tau} \} \bigr|^{p}$ in two cases.

\textbf{Step 2: The non-two-ends scenario}

Assume that $\beta \leq r^{\epsilon^{4}}$. For each ball $B_{v}$, we set
$$
g_{v}
=
\sum_{T\in \mathbb{T}_{\beta}\,:\,\exists J\in \mathcal{J}_{\lambda}(T)\text{ with }J\cap B_{v}\neq \varnothing} f_{T}.
$$
Then on $B_{v}$ we have
$$
\Bigl|\sum_{T\in \mathbb{T}_{\beta}}Ef_{T}\sum_{J\in \mathcal{J}_{\lambda}(T)}\mathbf{1}_{J}\Bigr|
\sim
|Eg_{v}|.
$$
Hence,
\begin{equation}
    \int_{X_{2}\cap B_{v}}|Br_{A_{2}}\{F_{\lambda,\beta}^{\tau}\}|^{p}
    \sim
    \int_{X_{2}\cap B_{v}}|Br_{A_{2}}Eg_{v}|^{p}.
\end{equation}

Observe that for every tube $T$, there are at most $\lesssim r^{\epsilon^{4}}$ balls $B_{v}$ for which there exists some $J\in \mathcal{J}_{\lambda}(T)$ satisfying $J\cap B_{v}\neq \varnothing$. It follows that
\begin{equation}\label{eq non-two L^{2}}
    \sum_{v}\|g_{v}\|_{2}^{2}
    \lesssim
    r^{\epsilon^{4}}\|g\|_{2}^{2}
    \lesssim
    r^{\epsilon^{4}}\|f\|_{2}^{2}.
\end{equation}

Since $A_{2}\gtrapprox A\geq r^{\epsilon^{20}}$, we also have
$
A_{2}\gtrapprox r^{(1-\epsilon^{2})\epsilon^{20}}.
$
{
We may therefore apply the induction hypothesis \eqref{eq broadinduction} on each ball $B_{v}$ of radius $r^{1-\epsilon^{2}}=R^{\frac{\epsilon^{2}}{(1-\epsilon^{2})^{k-1}}}$, which gives
\begin{equation}\label{eq non-two induction}
    \|Br_{A_{2}}Eg_{v}\|_{L^{p}(B_{v})}^{p}
    \leq
    C_{\epsilon}K^{O(1)} R^{\epsilon}r^{(1-\epsilon^{2})\epsilon}\|g_{v}\|_{L^{2}}^{2p-4}
    \sup_{\theta:r^{-\frac{1-\epsilon^{2}}{2}}-\text{square}}
    \|g_{v,\theta}\|_{L^{2}_{avg(\theta)}}^{4-{p}}.
\end{equation}
}
Moreover, by $L^{2}$-orthogonality,
\begin{equation}\label{eq non-two relation}
    \sup_{\theta:r^{-\frac{1-\epsilon^{2}}{2}}-\text{square}}
    \|g_{v,\theta}\|_{L^{2}_{avg(\theta)}}
    \lesssim
    \sup_{\theta: r^{-\frac{1}{2}}-\text{square}}
    \|f_{\theta}\|_{L^{2}_{avg}(\theta)}.
\end{equation}

{ Finally, summing over all $B_{v}$ and combining \eqref{eq non-two L^{2}}, \eqref{eq non-two induction}, and \eqref{eq non-two relation}, we obtain
\begin{equation}
\begin{aligned}
   \int_{B_{r}}|Br_{A}Ef'|^{p}
   &\lessapprox
   \sum_{v}
   C_{\epsilon}K^{O(1)}R^{\epsilon}r^{(1-\epsilon^{2})\epsilon}
   \|g_{v}\|_{L^{2}}^{2p-4}
   \sup_{\theta:r^{-\frac{1-\epsilon^{2}}{2}}-\text{square}}
   \|g_{v,\theta}\|_{L^{2}_{avg(\theta)}}^{4-p}\\
   &\lesssim
   r^{-\epsilon^{3}+\epsilon^{4}}
   C_{\epsilon}K^{O(1)}R^{\epsilon}r^{\epsilon}
   \|f\|_{L^{2}}^{2p-4}
   \sup_{\theta:r^{-\frac{1}{2}}-\text{square}}
   \|f_{\theta}\|_{L^{2}_{avg}(\theta)}^{4-{p}}.
\end{aligned}
\end{equation}
This completes the induction.
}

\textbf{Step 3: The two-ends scenario}

Suppose $r^{\epsilon^{4}}\leq \beta\leq r^{\epsilon^{2}}$. For each $T\in \mathbb{T}_{\beta}$, we consider the shading
\[
Y(T)=\bigcup_{J\in \mathcal{J}_{\lambda}(T)}(J\cap X).
\]
We can verify that $(\mathbb{T}_{\beta},Y)$ is a tube family with $(\epsilon^{2},\epsilon^{4})$-two-ends and $\lambda\beta$-dense shading.

Applying Lemma \ref{lem incidence} to $(\mathbb{T}_{\beta},Y)$, we obtain a refinement $X_{3}\subset X$ such that $|X\setminus X_{3}|\leq r^{-\epsilon^{2}}|X|$ and
\begin{equation}\label{eq incidencebound}
    \sup_{Q\in X_{3}}\#\{T\in \mathbb{T}_{\beta}:Y(T)\cap Q\neq \varnothing\}\lessapprox r^{O(\varepsilon_{3})}\mu,
\end{equation}
where
\[
\mu=r^{2\epsilon^{2}}m^{\frac{3}{4}}(\lambda\beta)^{-\frac{3}{4}}r^{\frac{1}{4}}(r^{-1}\#\mathbb{T})^{\frac{1}{4}}.
\]

Since $|X_{2}|\gtrapprox|X|$, we have $|X_{2}\setminus X_{3}|\lessapprox r^{-\epsilon^{2}}|X_{2}|$. Let $X_{4}=X_{2}\cap X_{3}$; then $|X_{4}|\gtrapprox|X_{2}|$.

Because $\int_{Q}|Br_{A_{2}}\{F_{\lambda,\beta}^{\tau}\}|^{p}$ is the same up to a constant multiple for all $r^{\frac{1}{2}}$-balls $Q\subset X_{2}$, it follows that
\begin{equation}
    \int_{X_{2}}|Br_{A_{2}}\{F_{\lambda,\beta}^{\tau}\}|^{p}\lessapprox \int_{X_{4}}|Br_{A_{2}}\{F_{\lambda,\beta}^{\tau}\}|^{p}.
\end{equation}

For any ball $Q\subset X_{4}$, we can use \eqref{eq incidencebound} to control the multiplicity.

We apply Lemma \ref{lem broad2} and pigeonholing to find two $\frac{1}{K}$-transverse squares $\tau_{1},\tau_{2}\in \mathcal{M}$ such that
\begin{equation}
    \int_{X_{4}}|Br_{A_{2}}\{F_{\lambda,\beta}^{\tau}\}|^{p}\lesssim K^{O(1)}\int_{X_{4}}\prod_{j=1,2}\Bigl|\sum_{T\in \mathbb{T}_{\beta}[\tau_{j}]}\sum_{J\in \mathcal{J}_{\lambda}(T)}Ef_{T}(x)\mathbf{1}_{J}\Bigr|^{\frac{p}{2}},
\end{equation}
where $\mathbb{T}_{\beta}[\tau_{j}]=\bigcup_{\theta\subset \tau_{j}}\mathbb{T}_{\beta}(\theta)$.

For each $B_{v}$, let
\[
\mathbb{T}_{\beta,v}[\tau_{j}]=\{T\in \mathbb{T}_{\beta}[\tau_{j}]:\exists J\in \mathcal{J}_{\lambda}(T),\; J\cap B_{v}\neq \varnothing\}.
\]
Since $\beta\in [r^{\epsilon^{4}},r^{\epsilon^{2}}]$, we compute that
\begin{equation}
    \int_{B_{r}}|Br_{A}Ef'|^{p}\lessapprox K^{O(1)}r^{10\epsilon^{2}}\max_{v}\int_{X_{4}\cap B_{v}}\prod_{j=1,2}\Bigl|\sum_{T\in\mathbb{T}_{\beta,v}[\tau_{j}]}Ef_{T}\Bigr|^{\frac{p}{2}}.
\end{equation}

For each $B_{v}$, we have
\[
\{T\in \mathbb{T}_{\beta}:Y(T)\cap Q\neq \varnothing\}=\{T\in \mathbb{T}_{\beta,v}:T\cap Q\neq \varnothing\},
\]
and thus
\begin{equation}
    \#\{T\in \mathbb{T}_{\beta,v}:T\cap Q\neq \varnothing\}\lessapprox r^{O(\varepsilon_{3})}\mu,
\end{equation}
for each $Q\subset X_{4}\cap B_{v}$.

We would like to apply the bilinear refined decoupling theorem. However, the
dominant pair of squares may only be transverse and need not form a good pair.
Then we reduce transverse pairs to good pairs. This rescaling reduction incurs only
an acceptable $K^{O(1)}$-loss.

In Lemma \ref{lem broad2}, we denote the dominant pair of $\frac{1}{K}$-squares as $\tau_{1},\tau_{2}$. The distance between these two $1/K$-squares is at least $\rho$. We assume the distance is $\rho$. We now perform an isotropic rescaling so that the distance between the two squares is $\sim 1$. Under the new coordinates, the two squares become $\frac{K^{\frac{1}{M}}}{K}$-scale squares $\tilde{\tau}_{1}, \tilde{\tau}_{2}$. We divide each of $\tilde{\tau}_{1}$ and $\tilde{\tau}_{2}$ into $K^{\frac{1}{M}}$ squares of scale $1/K$, denoted by $\tau_{1}'$ and $\tau_{2}'$, respectively. Each pair $(\tau_{1}', \tau_{2}')$ satisfies the transversality condition, and for each pair $(\tau_{1}', \tau_{2}')$, we apply the bilinear refined decoupling and rescale back. The process incurs an acceptable loss of $K^{O(1)}$ and we obtain
\begin{equation}\label{eq refineddecoupling}
    \int_{X_{4}\cap B_{v}}\prod_{j=1,2}\bigl|\sum_{T\in \mathbb{T}_{\beta,v}[\tau_{j}]}Ef_{T}\bigr|^{2}\lessapprox K^{O(1)}R^{\epsilon/5}r^{O(\varepsilon_{3})}\mu \sum_{T\in \mathbb{T}_{\gamma,m}}\|Ef_{T}\|_{L^{4}(\omega_{B_{r}})}^{4}.
\end{equation}

We explain the process in detail. Suppose that the centers of the two $K^{-1}$-squares are located at
\[
(-\rho,0)
\qquad\text{and}\qquad
\left(\rho,\,0\right).
\]

We perform the rescaling
\[
\xi'=\rho^{-1}\xi,\qquad
\eta'=\rho^{-1}\eta,\qquad
\gamma'=\rho^{-2}\gamma,
\]
and correspondingly
\[
x'=\rho x,\qquad
y'=\rho y,\qquad
t'=\rho^{2}t.
\]
Then the rescaled graph is
\[
\bar h(\xi',\eta')
=
\rho^{-2}
h\left(\rho{\xi'},
        \rho{\eta'}\right)
=\xi'\eta'+\rho^{-2}\psi(\rho\xi',\rho\eta'),
\]
where $\psi$ satisfies the same small-coefficient condition as in \eqref{eq smallcoeff}, provided that $K$ is large enough. We denote the rescaled surface by $\bar{\mathcal S}$.

In particular, uniformly in $\alpha_1,\alpha_2$ and $K$,
\[
\partial_{\xi'\eta'}^2 \bar{h}
= 1 + O(\varepsilon_0),
\qquad
\partial_{\xi'\xi'}^2 \bar{h},
\ \partial_{\eta'\eta'}^2 \bar{h}
= O(\varepsilon_0).
\]
Hence
\[
\det D^2 \bar{h}
= -1 + O(\varepsilon_0),
\]
so the rescaled surface has uniformly nonvanishing negative Gaussian curvature.

Each original frequency $K^{-1}$-square becomes a rectangle $\widetilde\tau_j$ with side lengths
\[
\frac{K^{\frac{1}{M}}}{K}
\quad\text{in the $\xi'$-direction}
\qquad\text{and}\qquad
\frac{K^{\frac{1}{M}}}{K}
\quad\text{in the $\eta'$-direction}.
\]
We decompose each $\widetilde\tau_j$ into squares of side length $1/K$. After the rescaling, the two centers are located at $(-1,0)$ and $(1,0)$ respectively. 
The two-sided curvature lower bound is invariant under this isotropic rescaling.

Indeed, for corresponding points $\zeta_j=\rho\zeta_j'$,
the chain rule gives
$$
\begin{aligned}
D^2\bar h(\zeta_j')
&=D^2h(\zeta_j),\\
\nabla \bar h(\zeta_2')-\nabla \bar h(\zeta_1')
&=\rho^{-1}\bigl(\nabla h(\zeta_2)-\nabla h(\zeta_1)\bigr).
\end{aligned}
$$
Thus the unit projected intersection direction is unchanged, and
$$
\partial_{v'(\zeta_1',\zeta_2')}^2\bar h(\zeta_j')
=
\partial_{v(\zeta_1,\zeta_2)}^2h(\zeta_j),
\qquad j=1,2.
$$
Hence the same curvature lower bound $\iota$ holds throughout
the rescaled square pair.

Therefore, every pair of resulting squares $\tau_1'\subset\widetilde\tau_1$ and $\tau_2'\subset\widetilde\tau_2$ is an $\iota$-good pair, provided $\varepsilon_0$ is sufficiently small. In particular, the isotropic rescaling normalizes the distance between these two squares to $\sim 1$, while preserving uniformly the curvature lower bound required in Section~4. Hence the bilinear refined decoupling theorem can be applied.

There are only $K^{O(1)}$ possible pairs $(\tau_1',\tau_2')$. Let $\widetilde
\tau_j$ denote the image of $\tau_j$ under the above frequency rescaling, and
decompose each $\widetilde\tau_j$ into squares $\tau_j'$ of side length
$1/K$. Let $\mathcal A$ denote the corresponding physical change of
variables
\[
    \mathcal A(x,y,t)
    =
    \left(
       \rho x,
       \rho y,
        \rho^{2}t
    \right),
\]
and set
\[
    r'=\rho^{2}r.
\]
For $T\in\mathbb T_{\beta,v}[\tau_j]$, write
\[
    \widetilde T=\mathcal A(T),
\]
and let $\widetilde{\mathbb T}_{\beta,v}[\tau_j']$ be the collection of such
rescaled tubes with frequency support in $\tau_j'$.

Under the above rescaling, $\mathcal A(X_4\cap B_v)$ is contained in a rectangle
box of dimensions
\[
\left(
\rho r,
\rho r,
\rho^{2} r
\right),
\]
and this box can be covered by $K^{O(1)}$ balls of radius comparable to $r'$.

Moreover, each $r^{1/2}$-ball $Q\subset X_4\cap B_v$ is mapped to an anisotropic
ellipsoid $\mathcal A(Q)$ of dimensions
\[
\left(
\rho r^{\frac{1}{2}},
\rho r^{\frac{1}{2}},
\rho^{2} r^{\frac{1}{2}}
\right),
\]
which can be covered by $K^{O(1)}$ standard $(r')^{1/2}$-balls. We use these
balls as localization cells, with cutoffs still supported in the corresponding
ellipsoids $\mathcal A(Q)$. Similarly, each anisotropic tube $\widetilde T$ can
be covered by $K^{O(1)}$ standard scale-$r'$ tubes for the rescaled surface
$\bar{\mathcal S}$.

Therefore, by Cauchy--Schwarz, after paying a $K^{O(1)}$ loss, we may restrict to a single good pair $(\tau_1',\tau_2')$.
\begin{align*}
&\int_{\mathcal{A}({X_{4}\cap B_{v}})}\prod_{j=1,2}|\sum_{\tau_{j}'\subset \tilde \tau_{j}}\sum_{ \tilde T\in \tilde{\mathbb{T}}_{\beta,v}[\tau_{j}']} Ef_{\tilde T}|^{2}
\\& \lesssim K^{O(1)}\sum_{\tau_{1}'\subset  \tilde \tau_{1}}\sum_{\tau_{2}'\subset  \tilde \tau_{2}} \int_{\mathcal{A}({X_{4}\cap B_{v}})}\prod_{j=1,2}|\sum_{ \tilde T\in \tilde {\mathbb{T}}_{\beta,v}[\tau_{j}']} Ef_{\tilde T}|^{2}
\\&\lesssim K^{O(1)}\int_{\mathcal{A}({X_{4}\cap B_{v}})}\prod_{j=1,2}|\sum_{ \tilde T\in \tilde{\mathbb{T}}_{\beta,v}[\tau_{j}']} Ef_{\tilde T}|^{2}.    
\end{align*}

Denote by $X'$ the resulting weighted union
of standard $(r')^{1/2}$-balls, and by $\mathbb T'_{\beta,v}[\tau_j']$ the
corresponding standard scale-$r'$ tubes. Due to the above $K^{O(1)}$ coverings,
\begin{align*}
&\int_{\mathcal{A}({X_{4}\cap B_{v}})}\prod_{j=1,2}|\sum_{ \tilde T\in \tilde{\mathbb{T}}_{\beta,v}[\tau_{j}']} Ef_{\tilde T}|^{2}
\\
&\lesssim
K^{O(1)}
\int_{X'}
\prod_{j=1,2}
\left|
    \sum_{T'\in \mathbb T'_{\beta,v}[\tau_j']}
    E f_{T'}
\right|^2 .
\end{align*}

The rescaled functions have Fourier support in
$
    \mathcal N_{C 1/\rho^{2}r}
    \bigl(\bar{\mathcal S}_{\tau_j'}\bigr)
    =
    \mathcal N_{C/r'}
    \bigl(\bar{\mathcal S}_{\tau_j'}\bigr).
$
Since the incidence estimate transfers to the rescaled configuration,
we naturally have 
\[
    M'_j \lesssim K^{O(1)} r^{O(\epsilon_3)}\mu,
    \qquad j=1,2.
\]
where $M_j'$ denotes the maximal number of standard scale-$r'$ tubes from
$\mathbb T'_{\beta,v}[\tau_j']$ meeting any $(r')^{1/2}$-ball in $X'$.
Thus, we may apply bilinear refined decoupling at scale $r'$ to the $\iota$-good pair
$(\tau_1',\tau_2')$ and then rescale back to obtain
\begin{equation}\label{eq refineddecoupling}
    \int_{X_{4}\cap B_{v}}\prod_{j=1,2}\bigl|\sum_{T\in \mathbb{T}_{\beta,v}[\tau_{j}]}Ef_{T}\bigr|^{2}\lessapprox K^{O(1)}R^{\epsilon/5}r^{O(\varepsilon_{3})}\mu \sum_{T\in \mathbb{T}_{\gamma,m}}\|Ef_{T}\|_{L^{4}(\omega_{B_{r}})}^{4},
\end{equation}
where the loss $R^{\epsilon/5}$ is from $
C(\rho^2r;\iota)
\le \overline C(R;\iota)
\lesssim_{\epsilon,D}R^{\epsilon/5}
$ by \eqref{eq epsilon/5}.

Since all $\|f_{T}\|_{2}$ are comparable, we have the following inequality by Lemma \ref{lem TS} and H\"{o}lder's inequality
\begin{equation}
\sum_{T\in \mathbb{T}_{\gamma,m}}\|Ef_{T}\|_{L^{4}(\omega_{B_{r}})}^{4}\lesssim \sum_{T\in \mathbb{T}_{\gamma,m}}\|f_{T}\|_{L^{2}}^{4}\lesssim \Big(\frac{1}{\#\mathbb{T}}\Big)\|f\|_{L^{2}}^{4}.    
\end{equation}

Therefore, we obtain
\begin{equation}
     \int_{X_{4}\cap B_{v}}\prod_{j=1,2}\bigl|\sum_{T\in \mathbb{T}_{\beta,v}[\tau_{j}]}Ef_{T}\bigr|^{2}\lessapprox K^{O(1)}R^{\epsilon/5}r^{O(\varepsilon_{3})}\Big(\frac{\mu}{\#\mathbb{T}}\Big)\|f\|_{L^{2}}^{4}.
\end{equation}
Recall that $\mu=r^{2\epsilon^{2}}m^{\frac{3}{4}}(\lambda\beta)^{-\frac{3}{4}}r^{\frac{1}{4}}(r^{-1}\#\mathbb{T})^{\frac{1}{4}}$ and $O(\varepsilon_{3})\leq \epsilon^{2}$, $K=R^{\epsilon^{10}}\leq r^{\epsilon^{2}}$; then we compute that
\begin{equation}\label{eq interpolation1}
\begin{aligned}
    \int_{X_{4}\cap B_{v}}\prod_{j=1,2}\Bigl|\sum_{T\in \mathbb{T}_{\beta,v}[\tau_{j}]}Ef_{T}\Bigr|^{2}
    \lessapprox\;
    &K^{O(1)}R^{\epsilon/5} r^{O(\varepsilon_{3})} \Bigl(\frac{r^{2\epsilon^{2}} m^{\frac{3}{4}} (\lambda\beta)^{-\frac{3}{4}} r^{\frac{1}{4}} (r^{-1}\#\mathbb{T})^{\frac{1}{4}}}{\#\mathbb{T}}\Bigr) \\
    &\qquad\times \Bigl(\#\mathbb{T}^{\frac{1}{2}} m^{-\frac{1}{2}} r^{-\frac{1}{2}}\Bigr)^{\frac{3}{2}}
    \|f\|_{L^{2}}^{\frac{5}{2}} \sup_{\theta}\|f_{\theta}\|_{L^{2}_{\mathrm{avg}}(\theta)}^{\frac{3}{2}} \\
    &\lessapprox R^{\epsilon/5}r^{O(\epsilon^{2})}(\lambda \beta r)^{-\frac{3}{4}}\|f\|_{L^{2}}^{\frac{5}{2}} \sup_{\theta}\|f_{\theta}\|_{L^{2}_{\mathrm{avg}}(\theta)}^{\frac{3}{2}}.
\end{aligned}
\end{equation}
This gives the $L^4$ estimate. On the other hand, Cauchy--Schwarz and Lemma \ref{lem L2} give the $L^2$ estimate:
\begin{equation}\label{eq interpolation2}
    \int_{X_{4}\cap B_{v}}\prod_{j=1,2}\bigl|\sum_{T\in \mathbb{T}_{\beta,v}}Ef_{T}\bigr|\lesssim (\lambda \beta r)\|f\|_{L^{2}}^{2}.
\end{equation}
Interpolating the inequalities \eqref{eq interpolation1} and \eqref{eq interpolation2}, we obtain 
\begin{equation}
\int_{B_{r}}|Br_{A}Ef'|^{\frac{22}{7}}\lessapprox R^{4\epsilon/35}r^{O(\epsilon^{2})}\|f\|_{L^{2}}^{\frac{16}{7}}\sup_{\theta}\|f_{\theta}\|_{L^{2}_{\mathrm{avg}}(\theta)}^{\frac{6}{7}}.
\end{equation}
This completes the proof of the broad part.

\end{proof}

\subsection{The proof of Theorem \ref{thm restriction2}}

We apply induction on scales for $R$ in the narrow case, while in the broad case we invoke Proposition \ref{prop broadestimate} to conclude the proof of Theorem \ref{thm restriction2}.
By interpolation, it suffices to consider the endpoint case  $p=\frac{22}{7}$. Recall that $\rho=K^{-1/M}$. Denote  $\mathcal{D}:=1/\rho$.
 
Let $\phi$ be a Schwartz function on $\mathbb{R}^{2}$ that equals $1$ on $B_{R^{2}}^{2}(0)$ and decays rapidly outside the ball. Take $\tilde f=f*\hat\phi$. Then we have
$$\Big|\int_{B_{R}}|Ef|^{p}-|E\tilde f|^{p}\Big|\lesssim R^{-1000}\|f\|_{2}^{p}.$$
Note that $\|\tilde f\|_{\infty}\leq R^{100}\|f\|_{L^{2}}$. Let $E_{u}$ be the level set $\{\tilde f\sim u \|f\|_{2}\}$ and let $E_{r}=\{|\tilde f|\leq R^{-100}\|\tilde f\|_{2}\}$.
The part of $E(\tilde f\mathbf{1}_{E_r})$ can be bounded straightforwardly.
Moreover, we can use dyadic pigeonholing to find a uniform dyadic number $u \in [R^{-100},R^{100}]$ such that 
$$\int_{B_{R}}|E\tilde f|^{p}\lessapprox \int_{B_{R}}|Eg|^{p},$$
where $g=\tilde f \mathbf{1}_{E_{u}}$. Take $K=R^{\epsilon^{10}}$. 
By \eqref{eq:broad-narrow-reduction} with $\epsilon'=\frac{\epsilon}{2m_{0}Mp}$ , we have
\begin{equation}\label{eq BN2}
    \int_{B_{R}}|Eg(x)|^{p}\lesssim_{\epsilon}K^{\frac{\epsilon}{M}}\sum_{N\in\mathcal N_\rho}\int_{B_{R}}|Eg_N(x)|^{p}+K^{2p}\cdot\int_{B_{R}} |Br_{K^{\epsilon'}}Eg|^{p}.
\end{equation}

{\it Case 1:}
Suppose the dominant narrow contribution comes from $\rho$-squares. Let $V$ be the $\frac{1}{\mathcal{D}}$-square that maximizes $\int_{B_{R}}|Eg_{V}|^p$. By a suitable affine transformation, we may assume that $V$ is centered at the origin, i.e., $V=\{(\xi,\eta):|\xi|,|\eta|\lesssim \frac{1}{\mathcal{D}}\}$. 

We perform the following change of variables
\[
\tilde \eta = \mathcal{D}\eta,\quad \tilde \xi = \mathcal{D}\xi,\quad \tilde y = \frac{y}{\mathcal{D}},\quad \tilde x = \frac{x}{\mathcal{D}},\quad \tilde t = \frac{t}{\mathcal{D}^{2}}.
\] 
Let $\tilde g(\tilde \xi, \tilde \eta)=g\bigl(\frac{\tilde \xi}{\mathcal{D}},\frac{\tilde \eta}{\mathcal{D}}\bigr)$. Then we have 
\begin{equation}\label{eq hhhh}
    \int_{B_{R}}|Eg_{V}|^{p}=\int_{B_{\frac{R}{\mathcal{D}}}}\Bigl|\int_{[-1,1]^{2}}e^{i\bigl[\tilde x\tilde \xi+\tilde y\tilde\eta+\tilde t \tilde \xi\tilde \eta+{\tilde t}{\mathcal{D}^{2}}\psi\bigl(\frac{\tilde \xi}{\mathcal{D}},\frac{\tilde \eta}{\mathcal{D}}\bigr)\bigr]}\, \tilde g \,d\tilde \xi \, d\tilde \eta\Bigr|^{p}\mathcal{D}^{-2p} \mathcal{D}^{4} \,d\tilde x\,d\tilde y \,d\tilde t,
\end{equation}
where $\psi$ denotes the remainder terms of $h$ other than $\xi\eta$. Note that
\[
\bigl\{(\tilde\xi,\tilde\eta,\tilde\xi\tilde\eta+\mathcal{D}^{2}\psi(\frac{\tilde\xi}{\mathcal{D}},\frac{\tilde\eta}{\mathcal{D}})):\tilde\xi,\tilde\eta\in [-1,1]\bigr\}
\]
is still a surface of the form \eqref{eq reduction}.
Therefore, we can use the induction hypothesis in Theorem \ref{thm restriction2} at the smaller scale $\frac{R}{\mathcal{D}}$ and obtain 
\begin{equation}
\begin{aligned}
    &\int_{B_{\frac{R}{\mathcal{D}}}}\Bigl|\int_{[-1,1]^{2}}e^{i\bigl[\tilde x\tilde \xi+\tilde y\tilde\eta+\tilde t \tilde \xi\tilde \eta+\tilde t \mathcal{D}^{2}\psi \bigl(\frac{\tilde \xi}{\mathcal{D}},\frac{\tilde \eta}{\mathcal{D}}\bigr)\bigr]}\, \tilde g \,d\tilde \xi \, d\tilde \eta\Bigr|^{\frac{22}{7}} \,d\tilde x\,d\tilde y \,d\tilde t \\
&\qquad \leq C_{\epsilon} R^{\frac{66}{7}\epsilon}\mathcal{D}^{-\frac{66}{7}\epsilon}\|\tilde g\|_{\frac{11}{4}}^{\frac{22}{7}} \leq C_{\epsilon}R^{\frac{66}{7}\epsilon}\mathcal{D}^{\frac{16}{7}-\frac{66}{7}\epsilon}\|g_{V}\|_{\frac{11}{4}}^{\frac{22}{7}},
\end{aligned}
\end{equation}
where we invoke the induction hypothesis for $(\frac{R}{\mathcal{D}})^{3\epsilon}$ in order to match the broad part.

Combine with \eqref{eq BN2} and \eqref{eq hhhh}, and we sum up all the contributions associated with $V$
\begin{equation}
    \int_{B_{R}}|Eg|^{\frac{22}{7}}\leq C_{\epsilon}R^{\frac{66}{7}\epsilon}K^{\frac{\epsilon}{M}}\mathcal{D}^{-\frac{66}{7}\epsilon}\sum_{V}\|g_{V}\|_{\frac{11}{4}}^{\frac{22}{7}}\lesssim C_{\epsilon}R^{\frac{66}{7}\epsilon}{K}^{-\frac{59\epsilon}{7M}}\|g\|_{\frac{11}{4}}^{\frac{22}{7}}.
\end{equation}

This closes the induction since  $K\gg1$.

{\it Case 2:}
Suppose the main parts of $\mathcal{N}_{\rho}$ are contained by $\rho$-pseudo bad strips. Let $S$ be the pseudo-bad strip that maximizes $\int_{B_{R}} |E g_{S}|^p$, and let $l$ be a pseudo-bad line contained in $S$. 
Recall that a pseudo-bad strip is a strip of length $1$ and width $\rho$ along whose long direction the second derivative of $h$ is bounded by $\varepsilon_0\rho$.
 We can apply an affine transformation so that the surface map becomes $h\circ F_{\alpha,a}^{-1}(\xi,\eta) = \eta \tilde h(\xi,\eta) + L(\xi)+P(\xi)$ and $\|\partial^{2}P\|_{\infty}\le \varepsilon_{0}\rho$. For simplicity, we slightly abuse the notation by using $h$ to denote $h\circ F_{\alpha,a}^{-1}$ and still use $S$ to denote the bad strip after the affine transformation. We have
$$S=\{(\xi,\eta):|\eta|\lesssim \frac{1}{\mathcal{D}}, |\xi|\lesssim 1\}.$$

Consider the extension operator
$$Eg_{S}(x,y,t)=\int_{[-1,1]^{2}}e^{i[x\xi+y\eta+t(\eta\tilde h(\xi,\eta)+L(\xi)+P(\xi))]}\,g(\xi,\eta)\,\mathbf{1}_{S}\,d\xi \,d\eta.$$

We perform the following coordinate transformation
\[
\tilde \eta = \mathcal{D}\eta,\quad \tilde \xi = \xi,\quad \tilde y = \frac{y}{\mathcal{D}},\quad \tilde x = x,\quad \tilde t = \frac{t}{\mathcal{D}}.
\] 
 Let $L(\xi)=c\xi$ and $\tilde g(\tilde \xi,\tilde \eta)=g\bigl(\tilde\xi,\frac{\tilde\eta}{\mathcal{D}}\bigr)$. Then we have
\begin{equation}
    Eg(x,y,t)=\frac{1}{\mathcal{D}}\int_{[-1,1]^{2}}e^{i[(\tilde x+\tilde t\mathcal{D} c)\tilde \xi+\tilde y \tilde \eta+\tilde t \tilde \eta \tilde h(\tilde \xi,\frac{\tilde \eta}{\mathcal{D}})+\tilde t \mathcal{D} P(\tilde \xi)]}\,\tilde g(\tilde \xi,\tilde \eta)\,d\tilde \xi\, d\tilde \eta=\frac{1}{\mathcal{D}}E_{\tilde{\mathcal{S}}}\tilde g(\tilde x+\mathcal{D}c\tilde t,\tilde y,\tilde t),
\end{equation}
where $\tilde{\mathcal{S}}=\bigl\{(\tilde \xi,\tilde \eta,\tilde \eta \tilde h(\tilde \xi, \tfrac{\tilde \eta}{\mathcal{D}})+\mathcal{D}P(\tilde\xi)) : \tilde \xi,\tilde \eta \in [-1,1] \bigr\}$ and the surface $\tilde {\mathcal{S}}$ still has the form of \eqref{eq reduction} by $\|\partial^{2}P\|_{\infty}\le \varepsilon_{0}\rho$.

Let the new coordinate transformation be $(\tilde x + \mathcal{D}c\tilde t,\ \tilde y,\ \tilde t) = (x', y', t')$, and one can verify that the Jacobian determinant of this transformation is $\sim 1$.
Therefore, we have
\begin{equation}\label{eq rescale}
\begin{aligned}
    \int_{B_{R}}|Eg_{S}|^{p} &\leq \mathcal{D}^{2-p}\int_{L}|E_{\tilde{\mathcal{S}}}\tilde g(\tilde x+\mathcal{D}c\tilde t,\tilde y, \tilde t)|^{p}\, d\tilde x\, d\tilde y\, d\tilde t \\
    &\sim \mathcal{D}^{2-p}\int_{L}|E_{\tilde{\mathcal{S}}}\tilde g(x',y',t')|^{p}\, dx'\, dy'\, dt',
\end{aligned}
\end{equation}
where $L=\{(\tilde x, \tilde y ,\tilde t): |\tilde x|\lesssim R,\ |(\tilde y, \tilde t)|\lesssim \tfrac{R}{\mathcal{D}}\}$ is a rectangular domain in space.

We divide $L$ into finitely overlapping $\frac{R}{\mathcal{D}}$-balls $\{B_{j}\}$. For each $B_{j}$, let $\tilde g_{j}$ be the sum of scale $\frac{R}{\mathcal{D}}$ wave packets associated with tubes intersecting $B_{j}$, so that 
\begin{equation}\label{eq narrow1}
    \int_{B_{j}}|E_{\tilde{\mathcal{S}}}\tilde g|^{p}\lesssim \int_{B_{j}}|E_{\tilde{\mathcal{S}}}\tilde g_{j}|^{p}+R^{-1000}\|g\|_{2}^{p}.
\end{equation}
Since the surface $\tilde{\mathcal{S}}$ still has the form of \eqref{eq reduction}, we can use the induction hypothesis in Theorem \ref{thm restriction2} at the smaller scale $\frac{R}{\mathcal{D}}$. Thus, we have
\begin{equation}\label{eq narrow2}
    \int_{B_{j}}|E_{\tilde{\mathcal{S}}}\tilde g_{j}|^{\frac{22}{7}}\leq C_{\epsilon} R^{\frac{66}{7}\epsilon}\mathcal{D}^{-\frac{66}{7}\epsilon} \|\tilde g_{j}\|_{\frac{11}{4}}^{\frac{22}{7}},
\end{equation}
where we invoke the induction hypothesis for $(\frac{R}{\mathcal{D}})^{3\epsilon}$ in order to match the broad part.
Since the Fourier transforms of $\{\tilde g_{j}\}$ are contained in finitely overlapping $\frac{R}{\mathcal{D}}$-balls, we have 
\begin{equation}\label{eq narrow3}
\sum_{j}\|\tilde g_{j}\|_{\frac{11}{4}}^{\frac{22}{7}}\lesssim \Bigl(\sum_{j}\int |\tilde g_{j}|^{\frac{11}{4}}\Bigr)^{\frac{8}{7}}\lesssim \|\tilde g\|_{\frac{11}{4}}^{\frac{22}{7}}   
\end{equation}
by the norm inequality $\|\cdot\|_{l^{\frac{22}{7}}}\leq \|\cdot\|_{l^{\frac{11}{4}}}$.
Combining \eqref{eq narrow1}, \eqref{eq narrow2}, and \eqref{eq narrow3}, we obtain
\begin{equation}
    \int_{L}|E_{\tilde{\mathcal{S}}}\tilde g|^{\frac{22}{7}}\lesssim C_{\epsilon} R^{\frac{66}{7}\epsilon}\mathcal{D}^{-\frac{66}{7}\epsilon} \|\tilde g\|_{\frac{11}{4}}^{\frac{22}{7}}=C_{\epsilon}R^{\frac{66}{7}\epsilon}\mathcal{D}^{\frac{8}{7}-\frac{66}{7}\epsilon}\| g_{S}\|_{\frac{11}{4}}^{\frac{22}{7}},
\end{equation}
Since the family $\{S\}$ has bounded overlap, combining \eqref{eq BN2} and \eqref{eq rescale} and summing over $S$ gives 
\begin{equation}
    \int_{B_{R}}|Eg|^{\frac{22}{7}}\lesssim C_{\epsilon} R^{\frac{66}{7}\epsilon}K^{\frac{\epsilon}{M}}\mathcal{D}^{-\frac{66}{7}\epsilon}\sum_{S}\|g_{S}\|_{\frac{11}{4}}^{\frac{22}{7}}\lesssim C_{\epsilon} R^{\frac{66}{7}\epsilon}{K}^{-\frac{59\epsilon}{7M}}\|g\|_{\frac{11}{4}}^{\frac{22}{7}}.
\end{equation}
Since $K\gg1 $, we close the induction in this case.

{\it Case 3}: 
{
Suppose the third part dominates. Since $|g|$ is essentially constant on its support, we have $\|g\|_{2}^{\frac{16}{7}}\|g\|_{\infty}^{\frac{6}{7}}\sim \|g\|^{\frac{22}{7}}_{\frac{11}{4}}\leq \|\tilde f\|^{\frac{22}{7}}_{\frac{11}{4}}$. Applying Proposition \ref{prop broadestimate}, there exists a constant $C_{\epsilon}$ such that
\begin{equation}\label{eq step3}
    \int_{B_{R}} |Br_{K^{\epsilon'}}Eg|^{p}\leq C_{\epsilon}K^{O(1)}R^{2\epsilon}\| g\|_{2}^{\frac{16}{7}}\sup_{\theta}\| g_{\theta}\|_{L^{2}_{\mathrm{avg}}(\theta)}^{\frac{6}{7}}.
\end{equation}
Therefore, we have
$$\int_{B_{R}}|Eg|^{p}\leq C_{\epsilon}R^{3\epsilon}\|f\|_{\frac{11}{4}}^{\frac{22}{7}},$$
where the last inequality follows from $K=R^{\epsilon^{10}}$ and $\|\tilde f\|_{\frac{11}{4}}\leq\|f\|_{\frac{11}{4}}$.

}

\bibliography{reference}
\bibliographystyle{alpha}

\end{document}